\documentclass[11pt]{amsart}

\usepackage{amssymb}
\usepackage{graphicx}   
\usepackage{mathrsfs}
\usepackage{amsmath}
\usepackage{amsthm}
\usepackage{amsfonts}
\usepackage{latexsym}

\newtheorem{thm}{Theorem}[section]
\newtheorem{prop}[thm]{Proposition}
\newtheorem{lem}[thm]{Lemma}
\newtheorem{cor}[thm]{Corollary}

\theoremstyle{definition}

\theoremstyle{claim}

\theoremstyle{remark}
\newtheorem{remark}[thm]{Remark}

\numberwithin{equation}{section}

\DeclareMathOperator{\dist}{dist} 

\begin{document}


\title{Energy Concentration for Min-Max Solutions of the Ginzburg-Landau Equations on manifolds with $b_1(M)\neq 0$.}

\author{Daniel L. Stern}
\address{Department of Mathematics, Princeton University, 
Princeton, NJ 08544}
\email{dls6@math.princeton.edu}



\begin{abstract}
We establish a new estimate for the Ginzburg-Landau energies $E_{\epsilon}(u)=\int_M\frac{1}{2}|du|^2+\frac{1}{4\epsilon^2}(1-|u|^2)^2$ of complex-valued maps $u$ on a compact, oriented manifold $M$ with $b_1(M)\neq 0$, obtained by decomposing the harmonic component $h_u$ of the one-form $ju:=u^1du^2-u^2du^1$ into an integral and fractional part. We employ this estimate to show that, for critical points $u_{\epsilon}$ of $E_{\epsilon}$ arising from the two-parameter min-max construction considered by the author in previous work, a nontrivial portion of the energy must concentrate on a stationary, rectifiable $(n-2)$-varifold as $\epsilon\to 0$.
\end{abstract}

\maketitle







\section{Introduction}

\hspace{6mm} In \cite{Stern}, we observed that on any compact Riemannian manifold $(M^n,g)$, a simple two-parameter min-max procedure for the energies
\begin{equation}\label{edef}
E_{\epsilon}(u)=\int_Me_{\epsilon}(u):=\int_M\frac{|du|^2}{2}+\frac{(1-|u|^2)^2}{4\epsilon^2}
\end{equation}
on $W^{1,2}(M,\mathbb{C})$ can be used to produce nontrivial solutions $u_{\epsilon}\in C^{\infty}(M,\mathbb{C})$ of the Ginzburg-Landau equation
\begin{equation}\label{gle}
\Delta u_{\epsilon}=-\epsilon^{-2}(1-|u_{\epsilon}|^2)u_{\epsilon}.
\end{equation}

\hspace{6mm} Inspired by Guaraco's work on the Allen-Cahn min-max \cite{Gu} and the well known connection between Ginzburg-Landau functionals and the codimension two area functional (see, e.g., \cite{BBO}, \cite{BOSp}, \cite{LR} for some of the major results in this line), we began to investigate in \cite{Stern} the energy concentration of these min-max solutions in the limit $\epsilon\to 0$, with an eye to providing a p.d.e.-based alternative to Almgren's min-max construction (\cite{A}, \cite{P}) of stationary integral varifolds in codimension two.

\hspace{6mm} To this end, we considered in \cite{Stern} the energy growth of the min-max solutions $u_{\epsilon}$, and established bounds of the form
\begin{equation}\label{oldebounds}
C^{-1}|\log\epsilon|\leq E_{\epsilon}(u_{\epsilon})\leq C|\log\epsilon|
\end{equation}
for some $C(M)>0$. Then, by translating arguments of \cite{BBO} to the setting of compact manifolds, we observed that, when the first Betti number $b_1(M)=0$, for any family of solutions of (\ref{gle}) satisfying (\ref{oldebounds}), a subsequence of the normalized energy measures
\begin{equation}\label{mudef}
\mu_{\epsilon}:=\frac{e_{\epsilon}(u_{\epsilon})}{|\log\epsilon|}dv_g
\end{equation}
converges to (the weight measure of) a stationary, rectifiable $(n-2)$-varifold \cite{Stern}. Thus, when $b_1(M)=0$, we confirmed that min-max methods for the Ginzburg-Landau functional can be used to produce a nontrivial stationary rectifiable $(n-2)$-varifold--a result which Da Rong Cheng informed us he had obtained independently. In particular, for these topologies, our methods nearly recover Almgren's existence result in codimension two, up to the subtle problem of determining whether the density of the limiting varifold takes values in $\pi\cdot\mathbb{N}$.

\hspace{6mm} When $b_1(M)\neq 0$, however, we noted that one could produce sequences of solutions of (\ref{gle}) with energy growth like (\ref{oldebounds}) whose energy distributes evenly over $M$--that is, solutions whose energy blows up without concentrating \cite{Stern}. Intuitively, one expects to find stable solutions of (\ref{gle}) approximating the harmonic representative of each class in $[M:S^1]\cong H^1(M;\mathbb{Z})$ (see, e.g., \cite{Alme1}, \cite{JMZ} for results in this direction\footnote{As an aside, we remark that for compact, oriented $(M^n,g)$, the existence of local minimizers of $E_{\epsilon}$ lying near each harmonic $\phi\in C^{\infty}(M,S^1)$ follows from Proposition \ref{threshprop} of this paper.} on domains in $\mathbb{R}^n$), so that when $b_1(M)\neq 0$, energy blow-up of the form (\ref{oldebounds}) can in principle arise from solutions associated to classes in $[M:S^1]$ with degree growing like $|\log\epsilon|^{1/2}$.

\hspace{6mm} The key to understanding how energy blows up for a given family $u_{\epsilon}$ of solutions to (\ref{gle}) lies in the study of the one-forms
\begin{equation}\label{jdefintro}
ju_{\epsilon}:=u_{\epsilon}^*(r^2d\theta)=u_{\epsilon}^1du_{\epsilon}^2-u_{\epsilon}^2du_{\epsilon}^1
\end{equation}
and their Hodge decompositions 
\begin{equation}\label{juhodgeintro}
ju_{\epsilon}=d^*\xi_{\epsilon}+h_{\epsilon}.
\end{equation} 
(That $d^*ju_{\epsilon}=0$ is a simple consequence of (\ref{gle}); hence the triviality of the exact part of (\ref{juhodgeintro}).) For solutions lying near harmonic maps to $S^1$ of degree $\sim |\log\epsilon|^{1/2}$, one expects energy growth to be driven by the harmonic part $h_{\epsilon}$, in the sense that
\begin{equation}\label{hendom}
E_{\epsilon}(u_{\epsilon})-\frac{1}{2}\|h_{\epsilon}\|_{L^2}^2=o(|\log\epsilon|)\text{ as }\epsilon\to 0.
\end{equation}
If, by contrast, the term $\|h_{\epsilon}\|_{L^2}^2=o(|\log\epsilon|)$ as $\epsilon\to 0$, we can go through the arguments of \cite{BBO} to show that the energy concentrates on a stationary, rectifiable $(n-2)$-varifold. One of the striking observations of \cite{BOSp} is that, for solutions of the parabolic Ginzburg-Landau equations in $\mathbb{R}^n$, the $|h_{\epsilon}|^2$ term\footnote{Rather, its analog in the setting of \cite{BOSp}.} doesn't interact in an essential way with the rest of the energy, so that, roughly speaking, one can remove it to obtain a family of modified energy measures exhibiting the desired\footnote{I.e., in the parabolic setting of \cite{BOSp}, concentration to a codimension-two Brakke flow.} concentration behavior \cite{BOSp}. In Section 3 of this paper, we translate the stationary case of this result to our setting, proving:

\begin{thm}\label{vortconc1} Let $u_{\epsilon}$ be a family of solutions of (\ref{gle}) on a compact, oriented $(M^n,g)$, with $E_{\epsilon}(u_{\epsilon})=O(|\log\epsilon|)$ as $\epsilon \to 0$. We can then find harmonic maps $\phi_{\epsilon}\in C^{\infty}(M,S^1)$ such that, setting
$$\tilde{u}_{\epsilon}:=\phi_{\epsilon}^{-1}\cdot u_{\epsilon}$$
and
$$\nu_{\epsilon}:=\frac{e_{\epsilon}(\tilde{u}_{\epsilon})}{|\log\epsilon|}dv_g\in C^0(M)^*,$$
there exists a subsequence $\epsilon_j\to 0$ and a stationary, rectifiable $(n-2)$-varifold $V$ such that
\begin{equation}
\nu_{\epsilon_j}\to \|V\|\text{ weakly in }C^0(M)^*\text{ as }\epsilon_j\to 0
\end{equation}
and
\begin{equation}
\|V\|(M)=\lim_{\epsilon_j\to 0}\frac{E_{\epsilon_j}(u_{\epsilon_j})-\frac{1}{2}\|h_{\epsilon_j}\|_{L^2}^2}{|\log\epsilon_j|}.
\end{equation}
(Here, $h_{\epsilon}$ is the harmonic part of $ju_{\epsilon}$, as in (\ref{juhodgeintro}).)
\end{thm}

\begin{remark} It will follow from the proof of Theorem \ref{vortconc1} that one can also characterize $\|V\|$ as the limit of the measures $\frac{1}{2}\frac{|d^*\xi_{\epsilon_j}|^2}{|\log\epsilon_j|}dv_g$ associated to the co-exact part $d^*\xi_{\epsilon}$ of $ju_{\epsilon}$, so that we can define $V$ without reference to the auxiliary maps $\tilde{u}_{\epsilon}$.
\end{remark}

\begin{remark} For simplicity, we have chosen to state all of our results in the setting of oriented manifolds, so that we can employ Hodge decompositions liberally without comment. But each result of course yields information in the unoriented case as well, by lifting the solutions $u_{\epsilon}$ to the double cover.
\end{remark}

\hspace{6mm} For the family of solutions $u_{\epsilon}$ arising from the min-max construction of \cite{Stern}, we then establish an estimate of the form
\begin{equation}\label{vortestint}
\liminf_{\epsilon\to 0}\frac{E_{\epsilon}(u_{\epsilon})-\frac{1}{2}\|h_{\epsilon}\|_{L^2}^2}{|\log\epsilon|}\geq c(M)>0,
\end{equation}
so that, by Theorem \ref{vortconc1}, we obtain 

\begin{thm}\label{minmaxvort} For the solutions $u_{\epsilon}$ of (\ref{gle}) produced by the two-parameter min-max construction of \cite{Stern} on a compact, oriented $(M^n,g)$ of dimension $n\geq 2$, we can find a subsequence $\epsilon_j\to 0$ and a nontrivial stationary, rectifiable $(n-2)$ varifold $V$ such that
\begin{equation}
\frac{\frac{1}{2}|d^*\xi_{\epsilon_j}|^2}{|\log\epsilon_j|}dv_g\to \|V\|\text{ weakly in }C^0(M)^*\text{ as }\epsilon_j\to 0.
\end{equation}
\end{thm}
In particular, we remove the topological condition $b_1(M)=0$ of \cite{Stern}, to show that energy concentration in the min-max solutions of the Ginzburg-Landau equations produces a nontrivial stationary, rectifiable $(n-2)$-varifold on every compact Riemannian manifold.

\hspace{6mm} The main ingredient in the proof of (\ref{vortestint}) is a new lower bound for the Ginzburg-Landau energy of arbitrary maps $u\in W^{1,2}(M,D^2)$ in terms of the harmonic component $h_u$ of $ju=u^*(r^2d\theta)$. Specifically, letting $\Lambda$ denote the lattice of integral harmonic one-forms (i.e., those harmonic one-forms of the form $j\phi$ for harmonic maps $\phi: M\to S^1$), we show in Proposition \ref{threshprop} that (in the relevant energy regime)
\begin{equation}\label{bignewest}
E_{\epsilon}(u)\geq \frac{1}{2}(1-\epsilon^{\alpha})\|h_u\|_{L^2}^2+c|\log\epsilon|\cdot\dist(h_u,\Lambda)-\epsilon^{\alpha}
\end{equation}
for some $\alpha(n)\in (0,1)$ and $c(M)>0$. This gives us a lower bound on the energy walls separating the components of $W^{1,2}(M,S^1)$ inside $W^{1,2}(M,\mathbb{C})$ (the higher-dimensional analog of the ``threshold transition energies" studied by Almeida in dimension two \cite{Alme1}, \cite{Alme2}), which we use to show that a map $v$ for which $E_{\epsilon}(v)-\frac{1}{2}\|h_v\|^2$ is small relative to $|\log\epsilon|$ cannot maximize energy in any of the two-parameter families used in the min-max construction. 

\begin{remark} It is well known that, when $b_1(M)\neq 0$, the presence of local minimizers for $E_{\epsilon}$ associated to classes in $[M:S^1]$ also gives rise to a number of other critical points via one-parameter mountain pass constructions (see, e.g., \cite{Alme1},\cite{Alme2},\cite{AlmeBeth} for more on this in the two-dimensional setting). Though we don't delve into this here, our results will also give information about energy concentration for these solutions--which, intuition suggests, may correspond to the min-max $(n-2)$-varifolds associated to classes in 
$$\pi_1(\mathcal{Z}_{n-2}(M;\mathbb{Z}),\{0\})\cong H_{n-1}(M;\mathbb{Z})\cong [M:S^1]$$
via Almgren's constructions \cite{A}, \cite{P}.
\end{remark}

\section*{Acknowledgements} I would like to thank my advisor Fernando Cod\'{a} Marques for his constant encouragement and interest in this work. The author is partially supported by NSF grants DMS-1502424 and DMS-1509027.

\section{Lower Bounds for $E_{\epsilon}(u)$ From the Harmonic Form $h_u$}

\hspace{6mm} Let $(M^n,g)$ be a compact, oriented Riemannian manifold of dimension $n\geq 2$, and consider a collection $\gamma_1,\ldots,\gamma_k: S^1\to M$ of smooth, simple closed curves generating the torsion-free part of $H_1(M;\mathbb{Z})$. On the space $\mathcal{H}^1(M)$ of harmonic one-forms, it will be convenient to introduce the box-type norm $|\cdot|_b$ given by
\begin{equation}\label{normbdef}
|h|_b:=\max_{1\leq i\leq k}|\int_{\gamma_i}h|,
\end{equation}
with associated metric $\dist_b$. We will denote by $\Lambda\subset \mathcal{H}^1(M)$ the lattice of integral harmonic one-forms; i.e.,
\begin{eqnarray*}
\Lambda &:=&\{h\in \mathcal{H}^1(M)\mid \int_{\gamma_i}h\in 2\pi \mathbb{Z}\text{ for }1\leq i \leq k\}\\
&=&\{\phi^*(d\theta)\mid \phi \in C^{\infty}(M,S^1)\text{ a harmonic map to }S^1\}.
\end{eqnarray*}

\hspace{6mm} For any $u\in W^{1,2}(M,\mathbb{C})$, we denote by $ju$ the one-form
\begin{equation}\label{jdefgen}
ju:=u^1du^2-u^2du^1=u^*(r^2d\theta),
\end{equation}
and let $h_u$ be the harmonic part of $ju$ in the Hodge decomposition
\begin{equation}\label{hodgedecomp}
ju=d\psi+d^*\xi+h_u.
\end{equation}
With notation in place, we can now state the central estimate of this section:

\begin{prop}\label{threshprop} There exist positive constants $\epsilon_0(M)\in (0,1)$, $C(M)<\infty$, and $\alpha(n)\in (0,1)$ such that if $u\in W^{1,2}(M,\mathbb{C})$ satisfies $|u|\leq 1$ and
\begin{equation}\label{energbdhyp}
E_{\epsilon}(u)\leq \epsilon^{-1/2}
\end{equation}
for some $\epsilon\in (0,\epsilon_0)$, then 
\begin{equation}\label{mainlemmaest}
E_{\epsilon}(u)\geq \frac{1}{2}(1-\epsilon^{\alpha})\|h_u\|_{L^2}^2+C^{-1}|\log\epsilon|\cdot \dist_b(h_u,\Lambda)-\epsilon^\alpha.
\end{equation}
\end{prop}

We were inspired to search for an estimate of this type by the work of Almeida \cite{Alme1}, \cite{Alme2} (see also \cite{AlmeBeth}), in which it is shown that complex-valued maps on a two-dimensional annulus with suitably bounded $E_{\epsilon}$ can be assigned a generalized degree, and that the minimum energy needed to connect two maps of different degrees is $\pi|\log\epsilon|$ to leading order as $\epsilon\to 0$. The estimate (\ref{mainlemmaest}) provides some extension of these results to higher dimensions, where now the nearest point in $\Lambda$ to $h_u$ takes on the role of degree, and we note that if $h_{u_0}$ and $h_{u_1}$ lie near different elements of $\Lambda$, then any path connecting $u_0$ to $u_1$ must pass through a map $v$ with $dist_b(h_v,\Lambda)=\pi$. We suspect that more precise estimates for the $|\log\epsilon|$ term in these higher-dimensional threshold transition energies will involve the masses of the min-max $(n-2)$-varifolds associated to classes in $H^1(M;\mathbb{Z})\cong \pi_1(\mathcal{Z}_{n-2}(M;\mathbb{Z}))$ by Almgren's work \cite{A}.

\begin{remark}\label{hrk} Note that while $u\mapsto ju$ does not define a continuous map from $W^{1,2}(M,\mathbb{C})$ into $L^2$ one-forms, it is evidently continuous as a map to the space of $L^1$ one-forms. And since projection onto the finite-dimensional subspace $\mathcal{H}^1(M)$ is continuous on the space of $L^1$ one-forms, it follows that 
$$u\mapsto h_u\text{ defines a continuous map }W^{1,2}(M,\mathbb{C})\to\mathcal{H}^1(M).$$
In particular, the quantities on the right-hand side of (\ref{mainlemmaest}) vary continuously with $u\in W^{1,2}$--a very simple but important observation which we will use without comment throughout the paper.
\end{remark}

\begin{proof} Since $M$ is oriented, for each $\gamma_i: S^1\to M$, we can choose an embedding
$$F_i: S^1\times B_1^{n-1}\to M$$
onto a tubular neighborhood of $\gamma_i(S^1)$ such that
\begin{equation}\label{fihom}
F_i|_{S^1\times 0}=\gamma_i
\end{equation}
and
\begin{equation}\label{flipbds}
Lip(F_i),Lip(F_i^{-1})\leq C(M).
\end{equation}

\hspace{6mm} Consider $u\in C^{\infty}(M,\mathbb{C})$ satisfying (\ref{energbdhyp}). By (\ref{flipbds}), we have
\begin{equation}
\int_{S^1\times B^{n-1}}e_{\epsilon}(u\circ F_i)\leq CE_{\epsilon}(u),
\end{equation}
so defining
\begin{equation}\label{gidef}
G_i:=\{y\in B^{n-1}\mid \int_{S^1\times y}e_{\epsilon}(u\circ F_i)\leq \frac{2CE_{\epsilon}(u)}{|B_1^{n-1}|}\},
\end{equation}
it follows from Fubini's theorem that
\begin{equation}\label{gimeasbd}
|G_i|\geq \frac{1}{2}|B_1^{n-1}|.
\end{equation}

\hspace{6mm} Writing
$$W(v)=\frac{1}{4}(1-|v|^2)^2,$$
for any $v\in C^{\infty}(S^1,\mathbb{C})$ with $|v|\leq 1$, we recall the standard computation 
\begin{eqnarray*}
\max_{S^1}W(v)&\leq &\int_{S^1}|d(W(v))|+\frac{1}{2\pi}\int_{S^1}W(v)\\
&\leq & \int_{S^1}|(1-|v|^2)||dv|+\frac{\epsilon^2}{2\pi}\int_{S^1}\frac{W(v)}{\epsilon^2}\\
&\leq &\int_{S^1}\epsilon \frac{1}{2}|dv|^2+2\epsilon\frac{W(v)}{\epsilon^2}+\frac{\epsilon^2}{2\pi}\frac{W(v)}{\epsilon^2}\\
&\leq &3\epsilon\int_{S^1}e_{\epsilon}(v).
\end{eqnarray*}
For every $y\in G_i$, it follows in particular that
$$\max_{S^1\times y}W(u\circ F_i)\leq C\epsilon^{1/2},$$
and consequently 
\begin{equation}\label{gammaubds}
\max_{S^1\times y}(1-|u\circ F_i|^2)\leq \frac{1}{2}
\end{equation}
provided $\epsilon^{1/2}\leq \frac{1}{16C}$.

\hspace{6mm} Consider now the Hodge decomposition
$$ju=d\psi+d^*\xi+h_u$$
of $ju$, and decompose $h_u$ further into its integral and fractional parts
$$h_u=j\phi+h',$$
where $\phi:M\to S^1$ is a harmonic map to $S^1$, and $h'$ satisfies
\begin{equation}\label{hprimedef}
|h'|_b=\dist_b(h_u,\Lambda)\leq \pi.
\end{equation}
(Note that this decomposition of $h_u$ is unique if and only if $|h'|_b<\pi$.)
For $y\in B_1^{n-1}$, denoting by $\gamma_{i,y}:S^1\to M$ the curve 
$$\gamma_{i,y}(\theta)=F_i(\theta,y),$$
we observe that, since $\gamma_{i,y}$ is homotopic to $\gamma_i$,
\begin{eqnarray*}
\int_{\gamma_{i,y}}ju&=&\int_{\gamma_{i,y}}(d^*\xi+h_u)\\
&=&\int_{\gamma_{i,y}}d^*\xi+\int_{\gamma_i}h_u\\
&=&\int_{\gamma_{i,y}}d^*\xi+\int_{\gamma_i}h'+2\pi \deg(\phi,\gamma_i),
\end{eqnarray*}
so that
\begin{equation}\label{distjuzest1}
\dist\left(\int_{\gamma_{i,y}}ju,2\pi\mathbb{Z}\right)=\dist\left(\int_{\gamma_{i,y}}d^*\xi+\int_{\gamma_i}h',2\pi\mathbb{Z}\right).
\end{equation}
On the other hand, if $y\in G_i$, then we can use (\ref{gammaubds}) to write
\begin{eqnarray*}
\int_{\gamma_{i,y}}ju&=&\int_{\gamma_{i,y}}|u|^2j(u/|u|)\\
&=&\int_{\gamma_{i,y}}\frac{(|u|^2-1)}{|u|^2}ju+2\pi \deg(u/|u|,\gamma_{i,y}),
\end{eqnarray*}
and since (by (\ref{gidef}) and (\ref{flipbds}))
$$\int_{\gamma_{i,y}}e_{\epsilon}(u)\leq CE_{\epsilon}(u),$$
we conclude that
\begin{eqnarray*}
\dist\left(\int_{\gamma_{i,y}}ju,2\pi\mathbb{Z}\right)&\leq& 2\int_{\gamma_{i,y}}(1-|u|^2)|du|\\
&\leq & \int_{\gamma_{i,y}}\epsilon |du|^2+4\epsilon\frac{W(u)}{\epsilon^2}\\
&\leq &C\epsilon E_{\epsilon}(u).
\end{eqnarray*}
Combining the preceding estimate with (\ref{distjuzest1}) and using the fact that 
$$|\int_{\gamma_i}h'|\leq |h'|_b\leq \pi,$$
it follows that
\begin{equation}\label{gammaixiest}
\int_{\gamma_{i,y}}|d^*\xi|\geq |\int_{\gamma_i}h'|-C\epsilon E_{\epsilon}(u)
\end{equation}
for every $y\in G_i$. Integrating (\ref{gammaixiest}) over $y\in G_i$ and using (\ref{flipbds}) to pass estimates between $S^1\times B^{n-1}$ and $F_i(S^1\times B^{n-1})$, we obtain an estimate of the form
\begin{equation}
\int_{F_i(S^1\times G_i)}|d^*\xi|\geq C^{-1}|\int_{\gamma_i}h'|-C\epsilon E_{\epsilon}(u).
\end{equation}
Choosing $i$ such that 
$$|\int_{\gamma_i}h'|=|h'|_b=dist_b(h_u,\Lambda),$$
we then arrive at the $L^1$ lower bound
\begin{equation}\label{dstarxil1bd}
\|d^*\xi\|_{L^1}\geq C^{-1}\dist_b(h_u,\Lambda)-C\epsilon E_{\epsilon}(u).
\end{equation}

\hspace{6mm} By H\"{o}lder's inequality, (\ref{dstarxil1bd}) evidently gives us a lower bound for $\|d^*\xi\|_{L^p}$ for any $p>1$, and applying the $L^p$ regularity for the Hodge Laplacian (see, e.g., \cite{Scott}) we obtain for $dju=dd^*\xi$ the $W^{-1,p}$ estimates
\begin{equation}
C(p,M)\|dju\|_{W^{-1,p}}\geq C^{-1}\dist_b(h_u,\Lambda)-C\epsilon E_{\epsilon}(u).
\end{equation}
In particular, fixing $p_n:=\frac{2n}{2n-1}=(2n)^*$, so that 
$$W^{-1,2}\hookrightarrow W^{-1,p_n}=(W^{1,2n})^*\hookleftarrow (C^{1/2})^*,$$ 
we record
\begin{equation}\label{chalfdualest}
\|dju\|_{W^{-1,p_n}}\geq C^{-1}\dist_b(h_u,\Lambda)-C\epsilon E_{\epsilon}(u),
\end{equation}
where $C=C(M)$.

\hspace{6mm} Next, by the fundamental estimates of Jerrard and Soner \cite{JSo} (see also \cite{BOSapp} for some related results and improved estimates when $n\geq 3$) we recall that for any $v\in C^{\infty}(M,\mathbb{C})$, 
\begin{equation}\label{jsest}
\|djv\|_{W^{-1,p_n}}=\|djv\|_{(W^{1,2n})^*}\leq C\left(\frac{E_{\epsilon}(v)}{|\log\epsilon|}+\epsilon^{\gamma}\right)
\end{equation}
for some $\gamma=\gamma(n)$ and $C=C(M)$. If we applied (\ref{jsest}) directly to the map $u$ in question, (\ref{chalfdualest}) would immediately yield the $|\log\epsilon|$ portion of the desired lower bound (\ref{mainlemmaest}), but would miss the $\|h_u\|_{L^2}^2$ part of the estimate. In order to bring the $\|h_u\|_{L^2}$ terms into the estimate, we will instead apply (\ref{jsest}) to the map
\begin{equation}\label{tildeudef}
\tilde{u}=\phi^{-1}\cdot u
\end{equation}
(where, recall, $\phi:M\to S^1$ is the harmonic map for which $j\phi$ gives the integral part of $h_u$). 

\hspace{6mm} For this modified map $\tilde{u}$, one checks directly that
\begin{equation}\label{jtildeu}
j\tilde{u}=ju-|u|^2j\phi
\end{equation}
and
\begin{equation}\label{enertildeu}
E_{\epsilon}(\tilde{u})=E_{\epsilon}(u)+\int_M\frac{1}{2}|u|^2|j\phi|^2-\langle ju,j\phi\rangle.
\end{equation}
By (\ref{jtildeu}), we see that 
$$\int\langle dj\tilde{u}-dju,\zeta\rangle=\int \langle -|u|^2j\phi, d^*\zeta\rangle$$
for any two-form $\zeta$. But since $j\phi$ is closed, we also have $\int\langle j\phi, d^*\zeta\rangle=0$, so in fact
\begin{eqnarray*}
\int\langle dj\tilde{u}-dju,\zeta\rangle&=&\int (1-|u|^2)\langle j\phi, d^*\zeta\rangle\\
&\leq &\|\zeta\|_{W^{1,2}}\|j\phi\|_{\infty}\left(\int (1-|u|^2)^2\right)^{1/2}\\
&\leq &\|\zeta\|_{W^{1,2}}\|j\phi\|_{\infty}\cdot 2\epsilon E_{\epsilon}(u)^{1/2}.
\end{eqnarray*}
By (\ref{energbdhyp}) and the harmonicity of $j\phi$, it then follows that
$$\|dj\tilde{u}-dju\|_{W^{-1,2}}\leq C\epsilon^{1/2} \|j\phi\|_{L^2},$$
and since\footnote{Using the fact that $|h'|_b\leq \pi$ by definition, and $\|\cdot\|_{L^2}\leq C|\cdot|_b$ on the finite-dimensional vector space $\mathcal{H}^1(M)$.}
$$\|j\phi\|_{L^2}\leq \|h_u\|_{L^2}+\|h'\|_{L^2}\leq \|h_u\|_{L^2}+C,$$
we conclude that
\begin{equation}\label{djdiffs}
\|dj\tilde{u}-dju\|_{W^{-1,2}}\leq C\epsilon^{1/2}\|h_u\|_{L^2}^2+C\epsilon^{1/2}.
\end{equation}
In particular, since $p_n=\frac{2n}{2n-1}\leq 2$, it follows from (\ref{djdiffs}) and (\ref{chalfdualest}) that
\begin{equation}
\|dj\tilde{u}\|_{W^{-1,p_n}}\geq C^{-1}\dist_b(h_u,\Lambda)-C\epsilon^{1/2}(1+\|h_u\|_{L^2}^2),
\end{equation}
and applying (\ref{jsest}) to $\tilde{u}$, we arrive at a bound of the form
\begin{equation}\label{mainetildeubd}
E_{\epsilon}(\tilde{u})\geq C^{-1}\dist_b(h_u,\Lambda)|\log\epsilon|-C\epsilon^{\beta}(1+\|h_u\|_{L^2}^2),
\end{equation}
where $\beta(n):=\min\{\frac{1}{2},\gamma(n)\}$.

\hspace{6mm} Finally, we use (\ref{enertildeu}) to compute
\begin{eqnarray*}
E_{\epsilon}(u)&=&E_{\epsilon}(\tilde{u})+\int_M\langle ju,j\phi\rangle-\frac{1}{2}|u|^2|j\phi|^2\\
&=&E_{\epsilon}(\tilde{u})+\int_M|j\phi|^2-\frac{1}{2}|u|^2|j\phi|^2+\langle h',j\phi\rangle\\
&\geq &E_{\epsilon}(\tilde{u})+\int_M\frac{1}{2}|h_u|^2-\frac{1}{2}|h'|^2\\
&\geq &E_{\epsilon}(\tilde{u})+\frac{1}{2}\|h_u\|_{L^2}^2-C\dist_b(h_u,\Lambda),
\end{eqnarray*}
so that, by (\ref{mainetildeubd}), we have
$$E_{\epsilon}(u)\geq \frac{1}{2}(1-C\epsilon^{\beta})\|h_u\|_{L^2}^2+(C^{-1}|\log\epsilon|-C)\dist_b(h_u,\Lambda)-C\epsilon^{\beta}.$$
Setting $\alpha(n):=\frac{1}{2}\beta(n)$ and taking $\epsilon\leq\epsilon_0(M)$ sufficiently small, we arrive at the desired estimate 
\begin{equation}
E_{\epsilon}(u)\geq \frac{1}{2}(1-\epsilon^{\alpha})\|h_u\|_{L^2}^2+C^{-1}|\log\epsilon|\cdot \dist_b(h_u,\Lambda)-\epsilon^\alpha.
\end{equation}
\end{proof}

\section{The Energy Concentration Varifold for General Solutions}

\hspace{6mm} Consider now an arbitrary family of solutions $u_{\epsilon}$ of the Ginzburg-Landau equations
\begin{equation}\label{gleqns}
\Delta u_{\epsilon}=-\epsilon^{-2}(1-|u_{\epsilon}|^2)u_{\epsilon}
\end{equation}
on a compact, orientable $(M^n,g)$, satisfying an energy bound of the form
\begin{equation}\label{obvlogbd}
E_{\epsilon}(u_{\epsilon})\leq C|\log\epsilon|
\end{equation}
for small $\epsilon$. Let the one-form $ju_{\epsilon}$ have Hodge decomposition
\begin{equation}
ju_{\epsilon}=d^*\xi_{\epsilon}+h_{\epsilon},
\end{equation}
and, as in the proof of Proposition \ref{threshprop}, decompose $h_{\epsilon}$ into an integral and fractional part
\begin{equation}
h_{\epsilon}=j\phi_{\epsilon}+h_{\epsilon}',
\end{equation}
where $\phi_{\epsilon}:M\to S^1$ is harmonic and $|h_{\epsilon}'|_b\leq \pi$. In this section, we establish the following concentration result (cf. Theorems A and B of \cite{BOSp}), from which the variant stated in Theorem \ref{vortconc1} follows immediately:

\begin{thm}\label{vortconc} With $u_{\epsilon}$ and $\phi_{\epsilon}$ as above, write 
$$\tilde{u}_{\epsilon}:=\phi_{\epsilon}^{-1}u_{\epsilon},$$
and let
$$T_{\epsilon}(\tilde{u}_{\epsilon}):=e_{\epsilon}(\tilde{u}_{\epsilon})Id-d\tilde{u}_{\epsilon}^* d\tilde{u}_{\epsilon}$$
denote the stress energy tensor of $\tilde{u}_{\epsilon}$ associated with the functional $E_{\epsilon}$. Then there exists a subsequence $\epsilon_j\to 0$ and a stationary, rectifiable $(n-2)$-varifold $V$ such that 
$$\frac{1}{|\log\epsilon_j|}T_{\epsilon_j}(\tilde{u}_{\epsilon_j})\to V$$
as generalized $(n-2)$-varifolds in the sense of \cite{AS}. Moreover, the mass $\|V\|(M)$ of this varifold is given by
\begin{equation}\label{vortmass}
\|V\|(M)=\lim_{\epsilon_j\to 0}\frac{1}{|\log\epsilon_j|}(E_{\epsilon_j}(u_{\epsilon_j})-\frac{1}{2}\|h_{\epsilon_j}\|_{L^2}^2).
\end{equation}
\end{thm}

\hspace{6mm} We recall from \cite{AS} that a generalized $m$-varifold on $M$ is a nonnegative Radon measure on the compact subbundle
$$A_m(M):=\{S\in End(TM)\mid S=S^*,-nId\leq S\leq Id,\text{ }tr(S)\geq m\}.$$
of $End(TM)$ consisting of symmetric endomorphisms with eigenvalues in $[-n,1]$ and trace $\geq m$. Just as for standard varifolds (see, e.g., \cite{All},\cite{Si}), the weight measure $\|V\|$ of a generalized $m$-varifold $V$ is the Radon measure on $M$ given by the pushforward of $V$ under the projection $A_m(M)\to M$, and the first variation $\delta V$ is the functional on $C^1$ vector fields defined by
\begin{equation}
\delta V(X)=\int_{A_m(M)}\langle S,\nabla X\rangle dV(S);
\end{equation}
$V$ is said to be stationary if $\delta V=0$. In the proof of Theorem \ref{vortconc} (as in \cite{BBO}, \cite{BOSp}, and \cite{Stern}), we will rely on the following measure-theoretic result of \cite{AS}:
\begin{prop}[\cite{AS}]\label{asprop} If $V$ is a generalized $m$-varifold for which $|\delta V(X)|\leq C\cdot |X|_{C^0}$, and $\Theta^m(\|V\|,x)\geq \eta$ at every $x\in spt(V)$ for some $\eta>0$, then there is a rectifiable $m$-varifold $\tilde{V}$ with $\|\tilde{V}\|=\|V\|$ and $\delta \tilde{V}=\delta V$.
\end{prop}

To apply Proposition \ref{asprop} in our setting (as in \cite{AS},\cite{BBO},\cite{BOSp},\cite{Stern}), we identify the tensors $T_{\epsilon}(\tilde{u})$ with elements of $C^0(A_{n-2}(M))^*$ as follows: observe that at a point $p\in M$ with $e_{\epsilon}(\tilde{u})(p)\neq 0$, the tensor
$$S(p)=Id-e_{\epsilon}(\tilde{u})^{-1}d\tilde{u}^*d\tilde{u}$$
defines a symmetric endomorphism of $TM$ with 
$$tr(S)=n-\frac{|du|^2}{e_{\epsilon}(u)}\geq n-2$$
and
$$-|X|^2\leq\langle SX,X\rangle=|X|^2-\frac{|\langle du,X\rangle|^2}{e_{\epsilon}(u)}\leq |X|^2,$$
so that $S(p)\in A_{n-2}(M)$. Thus, for $f\in C^0(A_{n-2}(M))$, we can set
\begin{equation}
\langle T_{\epsilon}(\tilde{u}),f\rangle:=\int_Me_{\epsilon}(\tilde{u})f(S(p))dv_g
\end{equation}
so that $T_{\epsilon}(\tilde{u})$ defines a generalized $(n-2)$-varifold with weight measure
\begin{equation}
e_{\epsilon}(\tilde{u})dv_g.
\end{equation}

Unlike the stress-energy tensors of the original maps $u_{\epsilon}$ solving (\ref{gleqns}), the tensors $T_{\epsilon}(\tilde{u}_{\epsilon})$ are not in general divergence-free, and therefore don't themselves define stationary generalized varifolds. However, as we'll see in the proof of Theorem \ref{vortconc}, the generalized $(n-2)$-varifolds $\frac{T_{\epsilon}(\tilde{u}_{\epsilon})}{|\log\epsilon|}$ will nonetheless have a stationary limit, to which we can apply Proposition \ref{asprop}. Furthermore, we remark that for a generalized varifold $V$ of the sort we're working with (which decomposes like a multiple of a Dirac mass in each fiber of $A_{n-2}(M)$), it follows directly from the arguments of \cite{AS} that the varifold $\tilde{V}$ constructed in Proposition \ref{asprop} is in fact equal to $V$.

\begin{proof} As in the proof of Proposition \ref{threshprop}, we note that $j\tilde{u}_{\epsilon}$ is given by
\begin{equation}
j\tilde{u}_{\epsilon}:=ju_{\epsilon}-|u_{\epsilon}|^2j\phi_{\epsilon},
\end{equation}
and use this to compute
\begin{eqnarray*}
d\tilde{u}_{\epsilon}^*d\tilde{u}_{\epsilon}&=&d|u_{\epsilon}|\otimes d|u_{\epsilon}|+|u_{\epsilon}|^{-2}j\tilde{u}_{\epsilon}\otimes j\tilde{u}_{\epsilon}\\
&=&du_{\epsilon}^*du_{\epsilon}-ju_{\epsilon}\otimes j\phi_{\epsilon}-j\phi_{\epsilon}\otimes ju_{\epsilon}+|u_{\epsilon}|^2j\phi_{\epsilon}\otimes j\phi_{\epsilon}\\
&=&du_{\epsilon}^*du_{\epsilon}-j\phi_{\epsilon}\otimes j\phi_{\epsilon}-(1-|u_{\epsilon}|^2)j\phi_{\epsilon}\otimes j\phi_{\epsilon}\\
&&-(d^*\xi_{\epsilon}+h_{\epsilon}')\otimes j\phi_{\epsilon}-j\phi_{\epsilon}\otimes (d^*\xi_{\epsilon}+h_{\epsilon}')
\end{eqnarray*}
and
\begin{equation}\label{etildeucomp}
e_{\epsilon}(\tilde{u}_{\epsilon})=e_{\epsilon}(u_{\epsilon})-\frac{1}{2}|j\phi_{\epsilon}|^2-\langle d^*\xi_{\epsilon}+h_{\epsilon}',j\phi_{\epsilon}\rangle-\frac{1}{2}(1-|u_{\epsilon}|^2)|j\phi_{\epsilon}|^2,
\end{equation}
from which we obtain
\begin{eqnarray*}
T_{\epsilon}(\tilde{u}_{\epsilon})&=&T_{\epsilon}(u_{\epsilon})-(\frac{1}{2}|j\phi_{\epsilon}|^2Id-j\phi_{\epsilon}\otimes j\phi_{\epsilon})\\
&&-[\langle d^*\xi_{\epsilon}+h_{\epsilon}',j\phi_{\epsilon}\rangle Id-2(d^*\xi_{\epsilon}+h_{\epsilon})\odot j\phi_{\epsilon}]\\
&&-(1-|u_{\epsilon}|^2)[\frac{1}{2}|j\phi_{\epsilon}|^2Id-j\phi_{\epsilon}\otimes j\phi_{\epsilon}].
\end{eqnarray*}
Now, since $u_{\epsilon}$ solves (\ref{gleqns}), we have $div[T_{\epsilon}(u_{\epsilon})]=0$, and since $j\phi_{\epsilon}$ and $h_{\epsilon}'$ are harmonic, one checks directly that
$$div\left(\frac{1}{2}|j\phi_{\epsilon}|^2-j\phi_{\epsilon}\otimes j\phi_{\epsilon}\right)=0,$$
and
$$div\left(\langle h_{\epsilon}',j\phi_{\epsilon}\rangle Id-2h_{\epsilon}\odot j\phi_{\epsilon}\right)=0$$
as well. For any $C^1$ vector field $X$ on $M$, integration of these identities yields
\begin{eqnarray*}
\int_M\langle T_{\epsilon}(\tilde{u}_{\epsilon}),\nabla X\rangle&=&-\int_M\langle\langle d^*\xi_{\epsilon},j\phi_{\epsilon}\rangle-2d^*\xi_{\epsilon}\odot j\phi_{\epsilon},\nabla X\rangle\\
&&-\int_M(1-|u_{\epsilon}|^2)\langle \frac{1}{2}|j\phi_{\epsilon}|^2Id-j\phi_{\epsilon}\otimes j\phi_{\epsilon},\nabla X\rangle\\
&\leq &C\left(\|j\phi_{\epsilon}\|_{\infty}\|d^*\xi_{\epsilon}\|_{L^1}+\|j\phi_{\epsilon}\|_{\infty}^2\int_M(1-|u_{\epsilon}|^2)\right)\|\nabla X\|_{\infty}.
\end{eqnarray*}
By the harmonicity of $j\phi_{\epsilon}$, we know that
$$\|j\phi_{\epsilon}\|_{\infty}\leq C\|j\phi_{\epsilon}\|_{L^2},$$
and since the fractional part $h_{\epsilon}'$ of $h_{\epsilon}$ is uniformly bounded, it's clear that
\begin{equation}\label{jphibds}
\|j\phi_{\epsilon}\|_{L^2}^2\leq C(1+\|h_{\epsilon}\|_{L^2}^2)\leq C(1+\|ju_{\epsilon}\|_{L^2}^2)\leq C|\log\epsilon|,
\end{equation}
by (\ref{obvlogbd}). Using this in the preceding estimate, we arrive at
$$\frac{1}{\|\nabla X\|_{\infty}}\int_M\langle T_{\epsilon}(\tilde{u}_{\epsilon}),\nabla X\rangle\leq C|\log\epsilon|^{1/2}\|d^*\xi_{\epsilon}\|_{L^1}+C|\log\epsilon|\int_M(1-|u_{\epsilon}|^2).$$

\hspace{6mm}To control the $\int_M(1-|u_{\epsilon}|^2)$ term, we simply note that
$$\int_M(1-|u_{\epsilon}|^2)\leq C\left(\int_M\epsilon^2W(u_{\epsilon})\right)^{1/2}\leq C\epsilon E_{\epsilon}(u)^{1/2},$$
so that, by (\ref{obvlogbd}),
$$\frac{1}{\|\nabla X\|_{\infty}}\int_M\langle T_{\epsilon}(\tilde{u}_{\epsilon}),\nabla X\rangle\leq C|\log\epsilon|^{1/2}\|d^*\xi_{\epsilon}\|_{L^1}+C\epsilon|\log\epsilon|^{3/2}.$$
Finally, as in the proof of Proposition \ref{threshprop}, we employ the Jerrard-Soner estimate (\ref{jsest}) and the $L^p$ regularity of the Hodge Laplacian to estimate the co-exact term:
\begin{equation}\label{dstarxil1est}
\|d^*\xi_{\epsilon}\|_{L^1}\leq C\|d^*\xi_{\epsilon}\|_{L^{p_n}}\leq C\|dju_{\epsilon}\|_{W^{-1,p_n}}\leq C\left(\frac{E_{\epsilon}(u_{\epsilon})}{|\log\epsilon|}+\epsilon^{\gamma}\right)
\end{equation}
(where, as before, we've fixed some $p_n$--say $p_n=\frac{2n}{2n-1}$--between $1$ and $\frac{n}{n-1}$).
Appealing once more to the energy bounds (\ref{obvlogbd}), it follows from the preceding computations that
\begin{equation}\label{tildeiv}
\int_M\langle T_{\epsilon}(\tilde{u}_{\epsilon}),\nabla X\rangle\leq C|\log\epsilon|^{1/2}\|\nabla X\|_{\infty}
\end{equation}
for every smooth vector field $X$ on $M$.

\hspace{6mm} Next, integrating (\ref{etildeucomp}), we observe that
\begin{eqnarray*}
E_{\epsilon}(\tilde{u}_{\epsilon})&=&E_{\epsilon}(u_{\epsilon})-\frac{1}{2}\|j\phi_{\epsilon}\|_{L^2}^2-\int(\langle h_{\epsilon}',j\phi_{\epsilon}\rangle+\frac{1}{2}(1-|u_{\epsilon}|^2)|j\phi_{\epsilon}|^2)\\
&=&E_{\epsilon}(u_{\epsilon})-\frac{1}{2}\|h_{\epsilon}\|_{L^2}^2+O(\|j\phi_{\epsilon}\|_{L^2}\|h_{\epsilon}'\|_{L^2}+\|j\phi\|_{\infty}^2\epsilon \cdot E_{\epsilon}(u)^{1/2}+1),
\end{eqnarray*}
and consequently,
\begin{equation}\label{etildmassest}
E_{\epsilon}(\tilde{u}_{\epsilon})=E_{\epsilon}(u_{\epsilon})-\frac{1}{2}\|h_{\epsilon}\|_{L^2}^2+O(|\log\epsilon|^{1/2})
\end{equation}
as $\epsilon\to 0$. Letting $V_{\epsilon}$ denote the generalized $(n-2)$-varifold given by
\begin{equation}\label{tildevdef}
V_{\epsilon}:=\frac{1}{|\log\epsilon|}T_{\epsilon}(\tilde{u}_{\epsilon}),
\end{equation}
it follows from (\ref{etildmassest}) that the $V_{\epsilon}$ have uniformly bounded mass as $\epsilon\to 0$, so we can extract a subsequence $\epsilon_j\to 0$ such that $V_{\epsilon_j}$ converges (weakly in $C^0(A_{n-2}(M))^*$) to a generalized $(n-2)$-varifold $V$. 

\hspace{6mm} For any $C^1$ vector field $X$ on $M$, it then follows from (\ref{tildeiv}) that
\begin{eqnarray*}
\delta V(X)&=&\lim_{\epsilon_j\to 0}\int_{A_{n-2}(M)}\langle S,\nabla X\rangle dV_{\epsilon_j}(S)\\
&=&\lim_{\epsilon_j\to 0}\frac{1}{|\log\epsilon_j|}\int_M\langle T_{\epsilon_j}(\tilde{u}_{\epsilon_j}),\nabla X\rangle\\
&\leq &C\lim_{\epsilon_j\to 0}|\log\epsilon_j|^{-1/2}\|\nabla X\|_{\infty}\\
&=&0,
\end{eqnarray*}
so $V$ is indeed stationary. Thus, writing
$$\nu_{\epsilon}:=\frac{e_{\epsilon}(\tilde{u}_{\epsilon})}{|\log\epsilon|}dv_g$$
and 
$$\nu:=\|V\|=\lim_{\epsilon_j\to 0}\nu_{\epsilon_j},$$
once we exhibit some $\eta>0$ such that
\begin{equation}\label{densest}
\Theta^{n-2}(\nu,x)=\lim_{r\to 0}\frac{\nu(B_r(x))}{\omega_{n-2}r^{n-2}}\geq \eta\text{ for }x\in spt(\nu),
\end{equation}
we can apply Proposition \ref{asprop} to conclude that $V$ is a stationary, rectifiable $(n-2)$-varifold.

\hspace{6mm} We recall now one of the key tools in the study of energy concentration for Ginzburg-Landau solutions: the $\eta$-ellipticity (or $\eta$-compactness) theorem of \cite{LR} and \cite{BBO}:

\begin{prop} (cf. \cite{LR},\cite{BBO})\label{etaellip} There exist positive constants $\eta(M),r_0(M),$ and $\epsilon_0(M)$ such that if $u_{\epsilon}$ solves (\ref{gleqns}) on $M$ for some $\epsilon\in (0,\epsilon_0)$, then at any point $p\in M$ for which $|u_{\epsilon}(p)|\leq \frac{1}{2}$, we have
\begin{equation}
r^{2-n}\int_{B_r(p)}e_{\epsilon}(u_{\epsilon})\geq \eta \log(r/\epsilon).
\end{equation}
for every $r\in (\epsilon,r_0)$.
\end{prop}
In the interest of completeness, we've included an appendix to this paper in which we translate the arguments of \cite{BBO} to the setting of compact manifolds, to obtain the precise version of $\eta$-ellipticity stated above.

\hspace{6mm} Now, since 
$$e_{\epsilon}(\tilde{u}_{\epsilon})=e_{\epsilon}(u_{\epsilon})+\frac{1}{2}|u_{\epsilon}|^2|j\phi_{\epsilon}|^2-\langle ju_{\epsilon},j\phi_{\epsilon}\rangle,$$
on any geodesic ball $B_r(p)$, we see that
\begin{eqnarray*}
\int_{B_r(p)}e_{\epsilon}(\tilde{u}_{\epsilon})&\geq &\int_{B_r(p)}e_{\epsilon}(u_{\epsilon})+\int_{B_r(p)}\langle ju_{\epsilon},j\phi_{\epsilon}\rangle\\
&\geq &\int_{B_r(p)}e_{\epsilon}(u_{\epsilon})-\|j\phi_{\epsilon}\|_{\infty}\int_{B_r(p)}|ju_{\epsilon}|\\
&\geq &\int_{B_r(p)}e_{\epsilon}(u_{\epsilon})-C|\log\epsilon|^{1/2}\int_{B_r(p)}|ju_{\epsilon}|,
\end{eqnarray*}
Next, we note that, by the monotonicity formula for Ginzburg-Landau solutions (see formula (\ref{mono1}) in the appendix) and the energy bound (\ref{obvlogbd}), we have
$$\int_{B_r(p)}|du_{\epsilon}|^2\leq C r^{n-2}E_{\epsilon}(u_{\epsilon})\leq Cr^{n-2}|\log\epsilon|,$$
and as a consequence, 
\begin{eqnarray*}
\int_{B_r(p)}|ju_{\epsilon}|&\leq & |B_r(p)|^{1/2}\left(\int_{B_r(p)}|du_{\epsilon}|^2\right)^{1/2}\\
&\leq &Cr^{n/2}\cdot r^{\frac{n-2}{2}}|\log\epsilon|^{1/2}\\
&=&Cr^{n-1}|\log\epsilon|^{1/2}.
\end{eqnarray*}
Plugging this into the lower bound for $\int_{B_r(p)}e_{\epsilon}(\tilde{u}_{\epsilon})$ above, we find that
\begin{equation}\label{etildedens1}
r^{2-n}\int_{B_r(p)}e_{\epsilon}(\tilde{u}_{\epsilon})\geq r^{2-n}\int_{B_r(p)}e_{\epsilon}(u_{\epsilon})-Cr|\log\epsilon|
\end{equation}
for all $r<inj(M)$.

\hspace{6mm} In particular, if $p$ is a point at which $|u_{\epsilon}(p)|\leq \frac{1}{2}$, then for $\epsilon<\epsilon_0$ and $r\leq r_0$, we can combine (\ref{etildedens1}) with Proposition \ref{etaellip} to conclude that 
\begin{equation}\label{etildetaest}
r^{2-n}\int_{B_r(p)}e_{\epsilon}(\tilde{u}_{\epsilon})\geq \eta |\log (r/\epsilon)|-Cr|\log\epsilon|.
\end{equation}
Setting
$$r_1:=\min\{r_0,\frac{\eta}{4C}\}\text{ and }\eta_1:=\frac{\eta}{4},$$
we obtain from (\ref{etildetaest}) the following:

\begin{lem}\label{tildeetaellip} There exist positive constants $\epsilon_0(M)$, $r_1(M)$, $\eta_1(M)$ such that if $u_{\epsilon}$ solves (\ref{gleqns}) on $M$ for $\epsilon\in (0,\epsilon_0)$, if $|u_{\epsilon}(p)|\leq \frac{1}{2}$, then
\begin{equation}
r^{2-n}\int_{B_r(p)}e_{\epsilon}(\tilde{u}_{\epsilon})\geq \eta_1|\log\epsilon|
\end{equation}
for every $r\in (\epsilon^{1/2},r_1).$
\end{lem}

\hspace{6mm} Now, with $\eta_1$ as in Lemma \ref{tildeetaellip}, consider a point $x\in M$ at which
\begin{equation}
2^{n-2}\omega_{n-2}\Theta^{n-2}(\nu,x)=\lim_{r\to 0}(r/2)^{2-n}\nu(B_r(x))<\eta_1.
\end{equation}
We can then choose $r\in (0,r_1)$ such that
$$r^{2-n}\nu(B_{2r}(x))<\eta_1,$$
and therefore 
\begin{equation}
r^{2-n}\nu_{\epsilon_j}(B_{2r}(x))<\eta_1
\end{equation}
for $\epsilon_j$ sufficiently small. In particular, for every $p \in B_r(x)$, it follows that
$$r^{2-n}\nu_{\epsilon_j}(B_r(p))<\eta_1,$$
and thus, by Lemma \ref{tildeetaellip},
\begin{equation}\label{ulowbd}
|u_{\epsilon_j}|>\frac{1}{2}\text{ on }B_r(x).
\end{equation}

\hspace{6mm} The objective now is to use (\ref{ulowbd}) to show that $\nu(B_{r/2}(x))=0$, from which we'll deduce that 
\begin{equation}
\Theta^{n-2}(\nu,\cdot)\geq \frac{\eta_1}{2^{n-2}\omega_{n-2}}\text{ on }spt(\nu).
\end{equation}
For solutions $u_{\epsilon}$ satisfying (\ref{obvlogbd}), we observe that the estimates of \cite{BOSe} give us a bound of the form
$$\int_M|d|u||^2+\frac{W(u)}{\epsilon^2}\leq C,$$
and recall from \cite{Stern} the pointwise gradient estimate
$$|du_{\epsilon}|^2\leq \frac{C}{\epsilon^2}(1-|u_{\epsilon}|^2).$$
Together, these imply
\begin{equation}\label{boscons}
\int_M[e_{\epsilon}(u_{\epsilon})-\frac{1}{2}|ju_{\epsilon}|^2]=\int_M[\frac{1}{2}(1-|u_{\epsilon}|^2)|du_{\epsilon}|^2+\frac{1}{2}|d|u_{\epsilon}||^2+\frac{W(u_{\epsilon})}{\epsilon^2}\leq C;
\end{equation}
and since 
$$\nu_{\epsilon}:=\frac{e_{\epsilon}(\tilde{u}_{\epsilon})dv_g}{|\log\epsilon|}=\frac{[e_{\epsilon}(u_{\epsilon})+\frac{1}{2}|u_{\epsilon}|^2|j\phi_{\epsilon}|^2-\langle ju_{\epsilon},j\phi_{\epsilon}\rangle]dv_g}{|\log\epsilon|},$$
it follows that
\begin{eqnarray*}
\nu&=&\lim_{\epsilon_j\to 0}\frac{[\frac{1}{2}|ju_{\epsilon_j}|^2+\frac{1}{2}|u_{\epsilon_j}|^2|j\phi_{\epsilon_j}|^2-\langle ju_{\epsilon_j},j\phi_{\epsilon_j}\rangle]dv_g}{|\log\epsilon_j|}\\
&=&\lim_{\epsilon_j\to 0}\frac{\frac{1}{2}[|ju_{\epsilon_j}-j\phi_{\epsilon_j}|^2-(1-|u_{\epsilon_j}|^2)|j\phi_{\epsilon_j}|^2]dv_g}{|\log\epsilon_j|}\\
&=&\lim_{\epsilon_j\to 0}\frac{1}{2}\frac{|ju_{\epsilon_j}-j\phi_{\epsilon_j}|^2}{|\log\epsilon_j|}dv_g.
\end{eqnarray*}
(Where we've used once again the estimate $\int_M(1-|u_{\epsilon}|^2)|j\phi_{\epsilon}|^2\leq C\epsilon |\log\epsilon|^{3/2}$.)

\hspace{6mm} Next, incorporating the $\phi_{\epsilon}$ term into the arguments of \cite{BBO}, we consider the one-forms
$$\alpha_{\epsilon}:=\psi(|u_{\epsilon}|^2)ju_{\epsilon}-j\phi_{\epsilon},$$
where $\psi(t)$ is some fixed nonnegative function satisfying
$$\psi(t)=\frac{1}{t}\text{ for }t\geq \frac{1}{4},\text{ }\psi(t)=1\text{ for }t\leq \frac{1}{8},$$
so that
\begin{equation}\label{sec3dalpha}
d\alpha_{\epsilon}=d[j(u_{\epsilon}/|u_{\epsilon}|)]=0\text{ on }\{|u_{\epsilon}|\geq \frac{1}{2}\}.
\end{equation}
It also follows from the choice of $\psi$ that
$$|\psi(t)-1|\leq C(1-t)$$
for all $t \in [0,1]$, so that
\begin{equation}\label{l2alphadiff}
\int_M|\alpha_{\epsilon}-(ju_{\epsilon}-j\phi_{\epsilon})|^2\leq C\int (1-|u_{\epsilon}|^2)|ju_{\epsilon}|^2\leq C
\end{equation}
(where we've employed (\ref{boscons}) in the final bound). 

\hspace{6mm} The uniform $L^2$ bound (\ref{l2alphadiff}) on the difference $\alpha_{\epsilon}-(ju_{\epsilon}-j\phi_{\epsilon})$, together with our previous computations for $\nu$, give us the new characterization
\begin{equation}\label{nuchar}
\lim_{\epsilon_j\to 0}\frac{1}{2}\frac{|\alpha_{\epsilon_j}|^2}{|\log\epsilon_j|}dv_g
\end{equation}
of the limiting measure $\nu$. We note also that (\ref{l2alphadiff}) gives us $L^2$ bounds on each component in the Hodge decomposition of $\alpha_{\epsilon}-(ju_{\epsilon}-j\phi_{\epsilon})$; hence, letting
$$\alpha_{\epsilon}=d\varphi_{\epsilon}+d^*\zeta_{\epsilon}+H(\alpha_{\epsilon})$$
and noting that
$$ju_{\epsilon}-j\phi_{\epsilon}=d^*\xi_{\epsilon}+h_{\epsilon}',$$
we deduce that
\begin{equation}
\|d\varphi_{\epsilon}\|_{L^2}\leq C,
\end{equation}
\begin{equation}
\|H(\alpha_{\epsilon})\|_{L^2}\leq C+\|h_{\epsilon'}\|_{L^2}\leq C',
\end{equation}
and
\begin{equation}\label{coexactdiffs}
\|d^*\zeta_{\epsilon}-d^*\xi_{\epsilon}\|_{L^2}\leq C.
\end{equation}
In particular, since $\|d^*\zeta_{\epsilon}\|_{L^2}^2$ is the only unbounded part of $\|\alpha_{\epsilon}\|_{L^2}^2$, it then follows from (\ref{nuchar}) that
\begin{equation}\label{nuchar2}
\nu=\lim_{\epsilon_j\to 0}\frac{1}{2}\frac{|d^*\zeta_{\epsilon_j}|^2}{|\log\epsilon_j|}dv_g.
\end{equation}

\hspace{6mm} Now, since 
$$dd^*\zeta_{\epsilon}=d\alpha_{\epsilon},$$
it follows from (\ref{sec3dalpha}) that $d^*\zeta_{\epsilon}$ is harmonic on $\{|u_{\epsilon}|\geq \frac{1}{2}\}$. In particular, by (\ref{ulowbd}), $d^*\zeta_{\epsilon}$ must be harmonic on the ball $B_r(x)$, giving us an estimate of the form\footnote{See, e.g., (\ref{localsupest}) in the appendix.}
\begin{equation}\label{zetaharmcons}
\|d^*\zeta_{\epsilon}\|_{L^2(B_{r/2}(x))}\leq C(n,r)\|d^*\zeta_{\epsilon}\|_{L^1(B_r(x))}.
\end{equation}
Next, we use (\ref{coexactdiffs}) to estimate
$$\|d^*\zeta_{\epsilon}-d^*\xi_{\epsilon}\|_{L^1}\leq C\|d^*\zeta_{\epsilon}-d^*\xi_{\epsilon}\|_{L^2}\leq C;$$
and since we saw in (\ref{dstarxil1est}) that $\|d^*\xi_{\epsilon}\|_{L^1}\leq C,$ it follows that
$$\|d^*\zeta_{\epsilon}\|_{L^1}\leq C'.$$
Plugging this bound into (\ref{zetaharmcons}), we appeal finally to (\ref{nuchar2}) to conclude that
\begin{equation}
\nu(B_{r/2}(x))=\lim_{\epsilon_j\to 0}\frac{1}{2}\frac{\|d^*\zeta_{\epsilon_j}\|_{L^2(B_{r/2}(x))}^2}{|\log\epsilon_j|}\leq \lim_{\epsilon_j\to 0} \frac{C}{|\log\epsilon_j|}=0.
\end{equation}

\hspace{6mm} We've now shown that, at any point $x$ in the support of the stationary generalized $(n-2)$-varifold $V$, the density $\Theta^{n-2}$ has the lower bound
$$\Theta^{n-2}(\|V\|,x)\geq \frac{\eta_1}{\omega_{n-2}2^{n-2}}>0.$$
Thus, we can apply Proposition \ref{asprop} to conclude that $V$ is indeed a stationary, rectifiable $(n-2)$-varifold. The formula (\ref{vortmass}) for the mass follows immediately from (\ref{etildmassest}).
\end{proof}

\begin{remark} The alternative characterization
$$\|V\|=\lim_{\epsilon_j\to 0}\frac{1}{2}\frac{|d^*\xi_{\epsilon_j}|^2}{|\log\epsilon_j|}dv_g$$
of $V$ is an immediate consequence of (\ref{nuchar2}) and (\ref{coexactdiffs}).
\end{remark}

\section{Energy Concentration for the Min-Max Solutions}

\hspace{6mm} We recall now the special solutions $u_{\epsilon}$ of (\ref{gleqns}) constructed in \cite{Stern} by applying min-max methods for $E_{\epsilon}$ over the collection $\Gamma(M)$ of families 
$$F\in C^0(D^2,W^{1,2}(M,\mathbb{C}))\text{\hspace{3mm} }D^2\ni y\mapsto F_y\in W^{1,2}(M,\mathbb{C})$$
satisfying
$$F_y\equiv y\text{ for all }y\in \partial D^2.$$
By construction, each $u_{\epsilon}$ occurs as the limit in $W^{1,2}$ of a min-max sequence $v_{\epsilon}^j\in W^{1,2}(M,\mathbb{C})$ 
of the form
\begin{equation}\label{vmmdef1}
v_{\epsilon}^j=F^j_{y_j}
\end{equation}
for some families $F^j\in \Gamma(M)$ and $y_j\in D^2$ with 
\begin{equation}\label{mmdef2}
E_{\epsilon}(F^j_{y_j})=\max_{y\in D^2}E_{\epsilon}(F^j_y)\to c_{\epsilon}(M):=\inf_{F\in \Gamma(M)}\max_{y\in D^2}E_{\epsilon}(F_y)\text{ as }j\to\infty.
\end{equation}
Recall from \cite{Stern} that we can choose these families $F^j$ to satisfy the additional requirement that 
\begin{equation}\label{valuesindisk}
F^j_y\in W^{1,2}(M,D^2)\text{--i.e., }\|F_y^j\|_{\infty}\leq 1
\end{equation}
for every $y\in D^2$. 

\hspace{6mm} Consider the class of maps $\mathcal{C}_{\epsilon}\subset W^{1,2}(M,D^2)$ given by
\begin{eqnarray*}
\mathcal{C}_{\epsilon}:=\{v \mid & v=w_1\text{ for some path }t\mapsto w_t\text{ in }C^0([0,1],W^{1,2}(M,D^2))\\
&\text{ such that }w_0\equiv 1\in \mathbb{C}\text{ and }E_{\epsilon}(v)=\max_{t\in [0,1]}E_{\epsilon}(w_t)\}.
\end{eqnarray*}
It follows from (\ref{vmmdef1})-(\ref{valuesindisk})--taking, for instance, $w_t=F_{(1-t)\cdot 1+t\cdot y_j}$--that each $v_{\epsilon}^j$ belongs to $\mathcal{C}_{\epsilon}$. The desired energy estimate
\begin{equation}\label{mmenergconc}
\liminf_{\epsilon\to 0}\frac{E_{\epsilon}(u_{\epsilon})-\frac{1}{2}\|h_{\epsilon}\|_{L^2}^2}{|\log\epsilon|}\to 0
\end{equation}
will be a straightforward consequence of the following lemma:

\begin{lem}\label{cebdslem} There exist constants $C(M)<\infty,$ $c(M)>0$, and $\epsilon_0(M)>0$ such that if $\epsilon \in (0,\epsilon_0)$, then for every $v\in \mathcal{C}_{\epsilon}$ satisfying $E_{\epsilon}(v)\leq \epsilon^{-1/2}$, either\footnote{Recall the box-type norm $|\cdot|_b$ defined by (\ref{normbdef}).}
\begin{equation}\label{hvcase1}
|h_v|_b\leq \pi
\end{equation}
or 
\begin{equation}\label{hvcase2}
E_{\epsilon}(v)-\frac{1}{2}(1-C|\log\epsilon|^{-1})\|h_v\|_{L^2}^2\geq c|\log\epsilon|.
\end{equation}
\end{lem}

\hspace{6mm} If $v_{\epsilon}^j\in\mathcal{C}_{\epsilon}$ is a min-max sequence approximating $u_{\epsilon}$, then evidently
$$E_{\epsilon}(v_{\epsilon}^j)\leq E_{\epsilon}(u_{\epsilon})+1\leq C|\log\epsilon|<\epsilon^{-1/2}$$
provided $\epsilon$ is sufficiently small and $j$ sufficiently large. Passing to a further subsequence if necessary, it then follows from the lemma that either 
\begin{equation}\label{mmcase1}
\|h_{v_{\epsilon}^j}\|_{L^2}\leq C(M)\text{ for every j=1,\ldots}
\end{equation}
or
\begin{equation}\label{mmcase2}
E_{\epsilon}(v_{\epsilon}^j)\geq \frac{1}{2}(1-C|\log\epsilon|^{-1})\|h_{v_{\epsilon}^j}\|_{L^2}^2+c|\log\epsilon|\text{ for every j}.
\end{equation}
In the first case, taking the limit of (\ref{mmcase1}) as $j\to \infty$, we deduce that $$\|h_{u_{\epsilon}}\|_{L^2}^2\leq C,$$ and therefore
\begin{equation}
E_{\epsilon}(u_{\epsilon})-\frac{1}{2}\|h_{u_{\epsilon}}\|_{L^2}^2\geq E_{\epsilon}(u_{\epsilon})-C\geq c|\log\epsilon|-C,
\end{equation}
by the original lower bound (\ref{oldebounds}) for $E_{\epsilon}(u_{\epsilon})$. On the other hand, if (\ref{mmcase2}) holds, then  passing to the limit $j\to \infty$ yields
\begin{equation}
E_{\epsilon}(u_{\epsilon})-\frac{1}{2}\|h_{u_{\epsilon}}\|_{L^2}^2\geq c|\log\epsilon|-\frac{1}{2}C|\log\epsilon|^{-1}\|h_{u_{\epsilon}}\|_{L^2}^2,
\end{equation}
and since, by the upper bound in (\ref{oldebounds}),
$$\frac{1}{2}\|h_{u_{\epsilon}}\|_{L^2}^2\leq E_{\epsilon}(u_{\epsilon})\leq C|\log\epsilon|,$$
it follows that
\begin{equation}\label{finalenergest}
E_{\epsilon}(u_{\epsilon})-\frac{1}{2}\|h_{u_{\epsilon}}\|_{L^2}^2\geq c|\log\epsilon|-C
\end{equation}
in this case as well. 

\hspace{6mm} Thus, an estimate of the form (\ref{finalenergest}) must hold for $\epsilon$ sufficiently small, and dividing by $|\log\epsilon|$ and taking $\epsilon \to 0$, we arrive immediately at the desired estimate (\ref{mmenergconc}). Finally, combining (\ref{mmenergconc}) with the result of Theorem \ref{vortconc}, we obtain the conclusion of Theorem \ref{minmaxvort}: the concentration of energy for the min-max solutions produces a nontrivial stationary, rectifiable $(n-2)$-varifold.

\begin{proof} To prove Lemma \ref{cebdslem}, consider $v\in \mathcal{C}_{\epsilon}$ satisfying $E_{\epsilon}(v)\leq \epsilon^{-1/2}$, and suppose that (\ref{hvcase1}) doesn't hold, so that 
\begin{equation}\label{case2cons}
|h_v-j\phi|_b< \pi
\end{equation}
for some \emph{nontrivial} harmonic $\phi \in C^{\infty}(M,S^1)$. For any path $t\mapsto w_t$ in $C^0([0,1],W^{1,2}(M,D^2))$ connecting $w_0\equiv 1$ to $w_1=v$ with $E_{\epsilon}(w_t)\leq \epsilon^{-1/2}$, since 
$$|h_{w_1}-j\phi|_b=|h_v-j\phi|_b< \pi \text{ and }|h_{w_0}-j\phi|_b=|j\phi|_b\geq 2\pi,$$
there must be some $t_0\in [0,1]$ at which 
\begin{equation}\label{distt0}
\dist_b(h_{w_{t_0}},\Lambda)=|h_{w_{t_0}}-j\phi|_b=\pi.
\end{equation}

\hspace{6mm} Now, applying Proposition \ref{threshprop} to this $w_{t_0}$, we obtain the lower bound
\begin{equation}\label{ewtlowbd}
E_{\epsilon}(w_{t_0})\geq \frac{1}{2}(1-\epsilon^{\alpha})\|h_{w_{t_0}}\|_{L^2}^2+C^{-1}\pi |\log\epsilon|-\epsilon^{\alpha}.
\end{equation}
Moreover, (\ref{case2cons}) and (\ref{distt0}) also imply that
$$|h_{w_{t_0}}-h_v|_b\leq 2\pi,$$
so that
$$\|h_{w_{t_0}}-h_v\|^2_{L^2}\leq C'(M),$$
and consequently
\begin{eqnarray*}
\frac{1}{2}\|h_{w_{t_0}}\|_{L^2}^2&\geq &\frac{1}{2}(1-\delta)\|h_v\|_{L^2}^2-\frac{1}{2}(\delta^{-1}-1)C'
\end{eqnarray*}
for any $\delta>0$. Choosing
$$\delta=\frac{C\cdot C'}{\pi|\log\epsilon|}$$
and plugging this back into our lower bound (\ref{ewtlowbd}) for $E_{\epsilon}(w_{t_0})$, we arrive at
\begin{eqnarray*}
E_{\epsilon}(w_{t_0})&\geq &\frac{1}{2}(1-C''|\log\epsilon|^{-1})(1-\epsilon^{\alpha})\|h_v\|_{L^2}^2-\frac{1}{2}C^{-1}\pi|\log\epsilon|-\epsilon^{\alpha}\\
&\geq &\frac{1}{2}(1-C|\log\epsilon|^{-1})\|h_v\|_{L^2}^2-c|\log\epsilon|
\end{eqnarray*}
for $\epsilon<\epsilon_0(M)$ chosen sufficiently small. 

\hspace{6mm} We've now shown that
$$\max_{t\in [0,1]}E_{\epsilon}(w_t)\geq \frac{1}{2}(1-C|\log\epsilon|^{-1})\|h_v\|_{L^2}^2-c|\log\epsilon|$$
for every path $w_t\in W^{1,2}(M,D^2)$ from $1$ to $v$ with $E_{\epsilon}(w_t)\leq E_{\epsilon}(v)$, so the estimate (\ref{hvcase2}) follows from the definition of $\mathcal{C}_{\epsilon}$.
\end{proof}

\section{Appendix: $\eta$-Ellipticity on Manifolds}

\subsection{Preliminaries} Throughout this appendix, $(M^n,g)$ will be a compact, oriented, $n$-dimensional manifold whose sectional curvature $sec_M$ and injectivity radius $inj(M)$ satisfy 
\begin{equation}\label{geobounds}
|sec_M|\leq 1< inj(M).
\end{equation}
We make the trivial observation that if $g$ satisfies (\ref{geobounds}), then so does $c^2g$ for any $c>1$, so that any estimates we obtain under the assumption (\ref{geobounds}) will hold under dilation. We will assume, moreover, that $n\geq 3$, and simply remark that, as in the Euclidean setting (cf. \cite{BBO}), the two-dimensional case of the $\eta$-ellipticity result is a relatively simple consequence of the monotonicity formula and pointwise gradient estimates for Ginzburg-Landau solutions. 

\hspace{6mm} In this subsection, we collect for the convenience of the reader all the basic geometric estimates that we will need to extend the arguments of \cite{BBO} to the curved setting. 

\hspace{6mm} Given a unit geodesic ball $B_1(p)\subset M$, let $g_0$ denote the flat metric on $B_1(p)$ induced by the exponential map $\exp_p$, and let $\rho(x)=\dist(x,p)$. On $B_1(p)$, it follows from (\ref{geobounds}) and the Rauch comparison theorem that
\begin{equation}\label{c0bounds}
\frac{1}{2}g_0\leq\frac{\sin(\rho)}{\rho}g_0\leq g\leq \frac{\sinh(\rho)}{\rho}g_0\leq 2g_0,
\end{equation}
while the Hessian comparison theorem tells us that
\begin{equation}\label{hesscomp}
-\rho^2g\leq(1-\rho\coth(\rho))g\leq g-\frac{1}{2}Hess_g(\rho^2)\leq (1-\rho\cot(\rho))g\leq \rho^2g.
\end{equation}

Following the treatment of Green's functions in Chapter 4 of \cite{Aub}, choose a smooth, nonincreasing function $f(t)$ such that $f(t)=0$ for $t\geq 1$ and $f(t)=1$ for $t\leq \frac{1}{2}$, and denote by $H_p$ the approximate Green's function
$$H_p(x):=-[(n-2)\sigma_{n-1}]^{-1}f(dist(x,p))\cdot dist(x,p)^{2-n},$$
(where $\sigma_{n-1}$ is the area of the standard unit $(n-1)$-sphere). One can then use Hessian comparison (\ref{hesscomp}) to see that
\begin{equation}\label{laphest}
|\Delta H_p(x)|\leq C(n)dist(p,x)^{2-n},
\end{equation}
and since (by the usual Green's formula computation)
$$\varphi(p)=\int_{B_1(p)}H_p\Delta \varphi-\int_{B_1(p)}\varphi\Delta H_p$$
for every $\varphi \in C^{\infty}(M)$, it follows that
\begin{equation}\label{greenest}
\varphi(p)\leq \frac{-1}{(n-2)\sigma_{n-1}}\int_{B_1(p)}\frac{\Delta \varphi(x)}{dist(x,p)^{n-2}}+C(n)\int_{B_1(p)}\frac{\varphi(x)}{dist(x,p)^{n-2}}.
\end{equation}

\hspace{6mm} As an application of (\ref{greenest}), let $\xi\in\Omega^k(M)$ be a smooth $k$-form on $M$, and for $\delta>0$, set
$$\varphi_{\delta}=(\delta+|\xi|^2)^{1/2}.$$
We can use the Bochner formula to compute\footnote{Remark on notation: though we use $\Delta$ to denote the negative spectrum Laplacian on functions, our $\Delta_H$ denotes the (positive spectrum) Hodge Laplacian $dd^*+d^*d$ on forms.}
\begin{eqnarray*}
\Delta \varphi_{\delta}&=&\varphi_{\delta}^{-1}\frac{1}{2}\Delta |\xi|^2-\varphi_{\delta}^{-3}|\frac{1}{2}d|\xi|^2|^2\\
&=&\varphi_{\delta}^{-1}[-\langle \xi,\Delta_H\xi\rangle+\mathcal{R}(\xi,\xi)+|\nabla \xi|^2-\varphi_{\delta}^{-2}|\langle \xi,\nabla \xi\rangle|^2]\\
&\geq &\varphi_{\delta}^{-1}[-\langle \xi,\Delta_H\xi\rangle+\mathcal{R}(\xi,\xi)]\\
&\geq &-|\Delta_H\xi|-C(n)|\xi|,
\end{eqnarray*}
where in the last line we use (\ref{geobounds}) to control the curvature term $\mathcal{R}(\xi,\xi)$. Applying (\ref{greenest}) with $\varphi=\varphi_{\delta}$, we then obtain the estimate
$$\varphi_{\delta}(p)\leq C(n)\int_{B_1(p)}\frac{|\Delta_H\xi|(x)+\varphi_{\delta}(x)}{dist(x,p)^{n-2}},$$
and letting $\delta\to 0$, we conclude that
\begin{equation}\label{xigreen1}
|\xi|(p)\leq C(n)\int_{B_1(p)}\frac{|\Delta_H\xi|(x)+|\xi|(x)}{dist(x,p)^{n-2}}.
\end{equation}
Note, moreover, that for any $\beta \in (0,1]$, applying (\ref{xigreen1}) to the rescaled metric $\beta^{-2}g$ yields
\begin{equation}\label{xigreen2}
|\xi|(p)\leq C(n)\beta^{-2}\int_{B_{\beta}(p)}\frac{|\Delta_H\xi|(x)+|\xi|(x)}{dist(x,p)^{n-2}}.
\end{equation}

\hspace{6mm} Next, using (\ref{c0bounds}) to estimate
$$\|dist(y,x)^{n-2}\|_{L^{\frac{n-1}{n-2}}(B_{\beta}(y))}\leq C(n)\beta^{\frac{n-2}{n-1}},$$
and applying Young's inequality for convolutions (see \cite{Aub}, Section 3.7 for a precise statement in the manifold setting) to control the second term in (\ref{xigreen2}), we find that for any $r+\beta \leq 1$, 
\begin{eqnarray*}
\|\xi\|_{L^s(B_r(p))}&\leq &C(n)\beta^{-2}r^{n/s}\sup_{y\in B_r(p)}\int_{B_{\beta}(y)}\frac{|\Delta_H\xi|(x)}{dist(x,y)^{n-2}}\\
&&+C(n)\beta^{-\frac{n}{n-1}}\|\xi\|_{L^q(B_{r+\beta}(p))}
\end{eqnarray*}
whenever $\frac{1}{q}=\frac{1}{s}+\frac{1}{n-1}$.
In particular, for $r\leq \frac{1}{2}$, taking $\beta=\frac{r}{2(n-1)}$ and iterating this estimate $(n-1)$ times, starting from $s=\infty$ and $q=n-1$, we arrive at the local $L^{\infty}$ estimate
\begin{equation}\label{localsupest}
\|\xi\|_{L^{\infty}(B_r(p))}\leq \frac{C(n)}{r^2}\sup_{y\in B_{2r}(p)}\int_{B_r(y)}\frac{|\Delta_H\xi|(x)}{dist(x,y)^{n-2}}+\frac{C(n)}{r^n}\int_{B_{2r}(p)}|\xi|.
\end{equation}

\hspace{6mm} Taking $r=\frac{1}{2}$ in (\ref{localsupest}) at a point $p$ where $|\xi|$ is maximal, we obtain the simple global estimate
\begin{equation}\label{dumbglob}
\|\xi\|_{L^{\infty}(M)}\leq C(n)\left(\max_{y\in M}\int_M\frac{|\Delta_H\xi|(x)}{dist(x,y)^{n-2}}+\|\xi\|_{L^1(M)}\right).
\end{equation}
As a consequence, we deduce the existence of a constant $C(M)$ such that if $\xi$ is $L^2$-orthogonal to the space $\mathcal{H}^k$ of harmonic $k$-forms, then\footnote{If no such estimate held, we could find a sequence $\xi_j\perp\mathcal{H}^k$ with $\|\xi_j\|_{\infty}=1$ and $\sup_{y\in M}\int_M\frac{|\Delta_H\xi_j|(x)}{dist(x,y)^{n-2}}\to 0$. Then $\int |d^*\xi_j|^2+|d\xi_j|^2=\int \langle \xi_j,\Delta_H\xi_j\rangle\to 0$, while (\ref{dumbglob}) implies $\lim_{j\to\infty}\|\xi_j\|_{L^1(M)}\geq C^{-1}>0$, which is clearly impossible for $\xi_j\perp \mathcal{H}^k$.}
\begin{equation}\label{globperpest}
\|\xi\|_{L^{\infty}(M)}\leq C_M \sup_{y\in M}\int_M\frac{|\Delta_H\xi|(x)}{dist(x,y)^{n-2}}.
\end{equation}
Setting the notation
\begin{equation}\label{capadef}
A_k(M,g):=\sup\{\|\xi\|_{L^{\infty}(M)}/\max_{y\in M}\int_M\frac{|\Delta_H\xi|(x)}{dist(x,y)^{n-2}}\mid \xi \perp \mathcal{H}^k\}<\infty,
\end{equation}
we note that $A_k(M,g)$ is scale invariant--i.e., 
\begin{equation}\label{akscal}
A_k(M,c^2g)=A_k(M,g)\text{ for any }c>0.
\end{equation}

\hspace{6mm} Finally, consider a $k$-form $\omega\in \Omega^k(M)$, and denote by $H(\omega)$ its harmonic part. Letting $h_i$ be an $L^2$ orthonormal basis for $\mathcal{H}^k(M)$, we then have
$$H(\omega)=\int_M\langle \omega,h_i\rangle h_i,$$
and therefore
$$|H(\omega)|_{L^{\infty}(M)}\leq \|\omega\|_{L^1(M)}\cdot\|h_i\|_{L^{\infty}(M)}^2.$$
Setting
\begin{equation}
Q_k(M,g):=\sup\{\frac{\|h\|_{L^{\infty}}^2}{\|h\|_{L^2}^2}\mid h\in \mathcal{H}^k\}\cdot \sup_{y\in M}\int_Mdist(x,y)^{2-n},
\end{equation}
so that
$$\int_M\frac{|H(\omega)|}{dist(x,y)^{n-2}}\leq Q_k(M,g)\|\omega\|_{L^1},$$
it follows that if $\xi=\Delta_H^{-1}(\omega-H(\omega))$, 
\begin{equation}\label{globalgreenest}
\|\xi\|_{L^{\infty}(M)}\leq A_k\left(\sup_{y\in M}\int_M\frac{|\omega(x)|}{dist(x,y)^{n-2}}+Q_k\|\omega\|_{L^1(M)}\right).
\end{equation}
We note that $Q_k(M,g)$ scales like
\begin{equation}\label{qkscal}
Q_k(M,c^2g)=c^{2-n}Q_k(M,g),
\end{equation}
which together with the scale-invariance (\ref{akscal}) of $A_k$, tells us that the class of metrics satisfying
\begin{equation}\label{geobounds2}
A_k(M,g)+Q_k(M,g)\leq C_1
\end{equation}
for a given $C_1>0$ is closed under dilations $g\to c^2g$ for $c>1$.

\vspace{6mm}

\subsection{The Monotonicity Formula} Little effort is needed to extend the Euclidean Ginzburg-Landau monotonicity formula (one of the central tools of \cite{LR} and \cite{BBO}) to the manifold setting. Given $r\in (0,1)$ and a solution $u$ of 
\begin{equation}\label{gleqnapp}
\Delta u=-\epsilon^{-2}(1-|u|^2)u,
\end{equation}
one simply plugs the gradient vector field $X=\frac{1}{2}\nabla dist^2(p,\cdot)^2$ into the inner-variation equation
\begin{equation}\label{bdryiv}
\int_{B_r(p)}\langle e_{\epsilon}(u)Id-du^*du,\nabla X\rangle=\int_{\partial B_r}[e_{\epsilon}(u)Id-du^*du](X,\nu),
\end{equation}
and uses Hessian comparison (\ref{hesscomp}) to arrive at the estimate
\begin{equation}\label{mono1}
\frac{d}{dr}F_{\epsilon}(u,p,r)\geq r^{2-n}\int_{\partial B_r}|\frac{\partial u}{\partial\nu}|^2+2r^{1-n}\int_{B_r(p)}\frac{W(u)}{\epsilon^2},
\end{equation}
where
\begin{equation}\label{bigfedef}
F_{\epsilon}(u,p,r):=e^{\Lambda r^2}r^{2-n}\int_{B_r(p)}e_{\epsilon}(u)
\end{equation}
and $\Lambda(n)=\frac{1}{2}(n+1)$. Throwing away the $|\frac{\partial u}{\partial \nu}|$ term and integrating from $0$ to $s$, one concludes in particular that
\begin{equation}\label{mono2}
F_{\epsilon}(u,p,s)\geq 2\int_0^s r^{1-n}\left(\int_{B_r(p)}\frac{W(u)}{\epsilon^2}\right)dr,
\end{equation}
and, after integrating by parts on the right (using now the assumption that $n\geq 3$), we obtain the useful estimate
\begin{equation}\label{monogreen}
\frac{n}{2}F_{\epsilon}(u,p,s)\geq \int_{B_s(p)}dist(x,p)^{2-n}\frac{W(u)}{\epsilon^2}.
\end{equation}

\vspace{6mm}

\subsection{The ``$\delta$-Energy Decay" Lemma} In this section, we extend the so-called ``$\delta$-energy decay" principle of Bethuel-Brezis-Orlandi to the curved setting--namely, we prove (cf. Theorem 3 of \cite{BBO}) 

\begin{lem}\label{deltaenlem} Let $(M,g)$ be a compact, orientable $3\leq n$-dimensional manifold satisfying $(\ref{geobounds})$ and $(\ref{geobounds2})$ for some\footnote{Of course, that $(\ref{geobounds2})$ holds for \emph{some} $C_1>0$ is automatically true; we're just fixing a particular bound, since it will make it easier to state the estimates.} $C_1>0$, and let $u\in C^{\infty}(M,\mathbb{C})$ be a solution of the $\epsilon$-Ginzburg-Landau equation $(\ref{gleqnapp})$ for some $\epsilon\in (0,1]$. Then, on any geodesic ball $B_{\delta}(p)\subset M$ with $\delta \leq \frac{1}{32}$, an estimate of the following form holds:
\begin{eqnarray*}
\int_{B_{\delta}(p)}e_{\epsilon}(u)&\leq & C(n,C_1)\left(\delta^n+\int_{B_1(p)}\frac{W(u)}{\epsilon^2}\right)\int_{B_1(p)}e_{\epsilon}(u)\\
&&+C(n)\int_{B_1(p)}\frac{W(u)}{\epsilon^2}.
\end{eqnarray*}
\end{lem}
As an immediate consequence, since the conditions (\ref{geobounds}) and (\ref{geobounds2}) are preserved under dilation, we can apply the lemma to the metrics $\frac{1}{r^2}g$ for $r\in (0,1)$ to obtain
\begin{cor}\label{deltaencor} Let $(M^n,g)$ and $u$ satisfy the hypotheses of Lemma \ref{deltaenlem}. Then, for any $r\in [\epsilon,1]$ and $\delta\leq \frac{1}{32}$, we have
\begin{eqnarray*}
\int_{B_{\delta r}(p)}e_{\epsilon}(u)&\leq & C(n,C_1)\left(\delta^n+r^{2-n}\int_{B_r(p)}\frac{W(u)}{\epsilon^2}\right)\int_{B_r(p)}e_{\epsilon}(u)\\
&&+C(n)\int_{B_r(p)}\frac{W(u)}{\epsilon^2}.
\end{eqnarray*}
\end{cor}

\begin{proof} To prove Lemma \ref{deltaenlem}, as in \cite{BBO}, we begin by decomposing $e_{\epsilon}(u)$ into
\begin{equation}\label{edecompapp}
e_{\epsilon}(u)=\frac{1}{2}(1-|u|^2)|du|^2+\frac{1}{2}|ju|^2+\frac{1}{8}|d|u|^2|^2+\frac{W(u)}{\epsilon^2},
\end{equation}
and observe that $\int_{B_{1/8}(p)}[e_{\epsilon}(u)-\frac{1}{2}|ju|^2]$ can be easily estimated in terms of $\int_{B_1(p)}\frac{W(u)}{\epsilon^2}$. In particular, we claim that
\begin{equation}\label{nonjparts}
\int_{B_{1/8}(p)}[e_{\epsilon}(u)-\frac{1}{2}|ju|^2]\leq C(n)\int_{B_1(p)}\frac{W(u)}{\epsilon^2}.
\end{equation}
To see this, we first recall from \cite{Stern} the simple gradient estimate
\begin{equation}\label{bochgradest}
|du|^2\leq [\epsilon^{-2}+|Ric|_{\infty}](1-|u|^2)\leq [\epsilon^{-2}+C(n)](1-|u|^2),
\end{equation}
from which the pointwise estimate
\begin{equation}\label{sigmaduest}
\frac{1}{2}(1-|u|^2)|du|^2\leq \frac{1}{2}[\epsilon^{-2}+C(n)](1-|u|^2)^2\leq C(n)\frac{W(u)}{\epsilon^2}
\end{equation}
follows immediately. To estimate the $d|u|^2$ term, we introduce a nonincreasing cutoff function $\zeta(t)$ satisfying 
\begin{equation}\label{zetadef}
\zeta(t)=1\text{ for }t\leq \frac{1}{8}\text{ and }\zeta(t)=0\text{ for }t\geq \frac{1}{4}.
\end{equation}
Letting $\rho$ again denote $dist(p,\cdot)$, we use the relation
$$\frac{1}{2}\Delta (1-|u|^2)=\frac{(1-|u|^2)|u|^2}{\epsilon^2}-|du|^2$$
together with (\ref{bochgradest}) and (\ref{sigmaduest}) to compute
\begin{eqnarray*}
\int_M\zeta(\rho)^2|d(1-|u|^2)|^2&=&\int_M 2\zeta(\rho)\zeta'(\rho)(1-|u|^2)\langle d\rho,d|u|^2\rangle\\
&&-\int_M\zeta(\rho)^2(1-|u|^2)\Delta (1-|u|^2)\\
&\leq &\int_{B_1(p)} 2\|\zeta'\|_{\infty}(1-|u|^2)\zeta(\rho)|d|u|^2|+2(1-|u|^2)|du|^2\\
&\leq &\int_{B_1(p)}\frac{1}{2}\zeta(\rho)^2|d(1-|u|^2)|^2+2\|\zeta'\|_{\infty}^2(1-|u|^2)^2\\
&&+C(n)\int_{B_1(p)}\frac{W(u)}{\epsilon^2},
\end{eqnarray*}
from which we conclude that
$$\int_{B_{1/8}}|d|u|^2|^2\leq\int_M\zeta(\rho)^2|d(1-|u|^2)|^2\leq C(n)\int_{B_1(p)}\frac{W(u)}{\epsilon^2}.$$
Combining this with (\ref{sigmaduest}), we establish the claimed inequality (\ref{nonjparts}).

\hspace{6mm} We turn now to the problem of estimating the $\int_{B_{\delta}}|ju|^2$ term. As in \cite{BBO}, instead of working directly with $ju$, we will first find estimates of the desired form for the one-form
\begin{equation}
\alpha:=\phi(|u|^2)ju,
\end{equation}
where $\phi(t)$ is a nonnegative function satisfying
\begin{equation}\label{phidef}
\phi(t)=\frac{1}{t}\text{ for }t\geq \frac{3}{4}\text{, }\phi(t)=1\text{ for }t\leq \frac{1}{4},\text{ and }|\phi'|\leq 2.
\end{equation}
The benefit of working with $\alpha$ instead of $ju$ comes from the fact that
$$d\alpha=dj(u/|u|)=0\text{ when }|u|^2\geq \frac{3}{4},$$
so that
$$|d\alpha|\leq C(1-|u|^2)|du^1\wedge du^2|\leq C(1-|u|^2)|du|^2,$$
and therefore, by (\ref{sigmaduest}),
\begin{equation}\label{dalphawapp}
|d\alpha|\leq C(n)\frac{W(u)}{\epsilon^2}.
\end{equation}
To see that the desired estimates for $ju$ will follow from those for $\alpha$, simply note that, by definition of $\phi(t)$,
$$|\phi(t)-1|\leq C_0(1-t)\text{ for all }t\in [0,1],$$
and consequently,
\begin{equation}\label{appjualpha}
|ju-\alpha|^2\leq (1-\phi(|u|^2))^2|ju|^2\leq C_0(1-|u|^2)^2|du|^2\leq C(n)\frac{W(u)}{\epsilon^2},
\end{equation}
by (\ref{sigmaduest}).

\hspace{6mm} Let $\zeta(\rho)$ again be the radial cutoff function given by (\ref{zetadef}). Consider the two-form
$$\omega:=\zeta(\rho)d\alpha.$$
and define
\begin{equation}
\xi:=\Delta_H^{-1}(\omega-H(\omega)).
\end{equation}
By (\ref{dalphawapp}) and the choice of $\zeta$, we have the pointwise estimate
\begin{equation}\label{ptwiseomega}
|\omega|\leq \zeta(\rho)\frac{W(u)}{\epsilon^2}\leq \frac{W(u)}{\epsilon^2}\chi_{B_{1/4}(p)}
\end{equation}
so applying (\ref{globalgreenest}) and (\ref{geobounds2}) to $\xi$, it follows that
\begin{equation}\label{xisupestapp1}
\|\xi\|_{\infty}\leq C(n,C_1)\left(\sup_{y\in M}\int_{x\in B_{1/4}(p)}\frac{W(u)/\epsilon^2}{dist(x,y)^{n-2}}+\int_{B_{1/4}(p)}\frac{W(u)}{\epsilon^2}\right).
\end{equation}
Note that if $dist(y,p)>\frac{3}{8}$, then 
\begin{equation}
\int_{x\in B_{1/4}(p)}\frac{W(u)/\epsilon^2}{dist(x,y)^{n-2}}\leq 8^{n-2}\int_{B_{1/4}(p)}\frac{W(u)}{\epsilon^2},
\end{equation}
while if $dist(y,p)\leq \frac{3}{8}$, we have 
\begin{equation}\label{yintgreenest}
\int_{x\in B_{1/4}(p)}\frac{W(u)/\epsilon^2}{dist(x,y)^{n-2}}\leq \int_{B_{5/8}(y)}\frac{W(u)/\epsilon^2}{dist(x,y)^{n-2}}.
\end{equation}
The trick now (cf. \cite{BBO}) is to apply the monotonicity estimate (\ref{monogreen}) at the points $y\in B_{3/8}(p)$ to find that
\begin{equation}\label{monogreencons}
\int_{B_{5/8}(y)}\frac{W(u)/\epsilon^2}{dist(x,y)^{n-2}}\leq C(n)\int_{B_{5/8}(y)}e_{\epsilon}(u)\leq C(n)\int_{B_1(p)}e_{\epsilon}(u),
\end{equation}
so that, combining (\ref{xisupestapp1})-(\ref{monogreencons}), we arrive at 
\begin{equation}\label{mainxisupest}
\|\xi\|_{L^{\infty}(M)}\leq C(n,C_1)\int_{B_1(p)}e_{\epsilon}(u).
\end{equation}
With the estimates (\ref{mainxisupest}) and (\ref{ptwiseomega}) in hand, we compute
\begin{eqnarray*}
\int_M\langle \xi,\Delta_H\xi\rangle&=&\int_M\langle \xi,\omega-H(\omega)\rangle\\
&=&\int_M\langle \xi,\omega\rangle\\
&\leq &\|\xi\|_{\infty}\int_M|\omega|\\
&\leq &C(n,C_1)\left(\int_{B_1(p)}e_{\epsilon}(u)\right)\left(\int_{B_1(p)}\frac{W(u)}{\epsilon^2}\right),
\end{eqnarray*}
from which it follows that
\begin{equation}\label{appdstarxiest}
\|d^*\xi\|_{L^2(M)}^2\leq C(n,C_1)\left(\int_{B_1(p)}e_{\epsilon}(u)\right)\int_{B_1(p)}\frac{W(u)}{\epsilon^2}.
\end{equation}

\hspace{6mm} Next, let $\varphi\in C^{\infty}(M)$ solve 
$$\Delta\varphi=div(\zeta(\rho)\alpha),$$
so that $d\varphi$ gives the exact part of $\zeta(\rho)\alpha$, and consider also the solution $\psi \in C^{\infty}(M)$ of $$\Delta\psi=div(\zeta(\rho)ju).$$
It then follows from (\ref{appjualpha}) that
\begin{equation}\label{phipsidiff}
\int_M|d\varphi-d\psi|^2\leq \int_M\zeta(\rho)^2|ju-\alpha|^2\leq C(n)\int_{B_1(p)}\frac{W(u)}{\epsilon^2},
\end{equation}
and since $div(ju)=0$, we note that $\psi$ is harmonic on $B_{1/8}(p)$. In particular, we can apply (\ref{localsupest}) to $d\psi$ with $r=\frac{1}{32}$ to conclude that
\begin{equation}
\|d\psi\|_{L^{\infty}(B_{1/32}(p))}\leq C(n)\|d\psi\|_{L^1(B_1(p))}\leq C'(n)\|d\psi\|_{L^2(B_1(p))},
\end{equation}
so that, for any $\delta\leq \frac{1}{32}$,
\begin{equation}\label{psideltaballest}
\int_{B_{\delta}(p)}|d\psi|^2\leq C(n)\delta^n\|\zeta(\rho)ju\|_{L^2(M)}^2\leq C(n)\delta^n\int_{B_1(p)}e_{\epsilon}(u).
\end{equation}
Putting together (\ref{phipsidiff}) and (\ref{psideltaballest}), we find that
\begin{equation}\label{phideltaballest}
\int_{B_{\delta}(p)}|d\varphi|^2\leq C(n)[\delta^n\int_{B_1(p)}e_{\epsilon}(u)+\int_{B_1(p)}\frac{W(u)}{\epsilon^2}]
\end{equation}
whenever $\delta\leq \frac{1}{32}$.

\hspace{6mm} It remains to estimate the $L^2$ norm of the difference
$$\beta:=\zeta(\rho)\alpha-d\varphi-d^*\xi$$
on $B_{\delta}(p)$. Since $d^*d\varphi=d^*(\zeta(\rho)\alpha)$ by definition of $\varphi$, it's clear that
$$d^*\beta=0,$$
so that
\begin{eqnarray*}
\Delta_H\beta&=&d^*d\beta\\
&=&d^*d(\zeta(\rho)\alpha-d^*\xi)\\
&=&d^*[d(\zeta(\rho)\alpha)-dd^*\xi]-d^*[d^*d\xi+H(\zeta(\rho)d\alpha)]\\
&=&\zeta'(\rho)d\rho\wedge \alpha,
\end{eqnarray*}
by definition of $\xi$. In particular, it follows from the choice of $\zeta$ that $\beta$ is harmonic on $B_{1/8}(p)$, so we can again apply (\ref{localsupest}) on $B_{1/32}(p)$ to conclude that
\begin{equation}\label{betalinf1}
\|\beta\|_{L^{\infty}(B_{1/32}(p))}\leq C(n)\|\beta\|_{L^2(M)}.
\end{equation}
On the other hand, since $\beta+d^*\xi$ gives the exact part of $\zeta(\rho)\alpha$, we know that
$$\int_M|\beta+d^*\xi|^2\leq \int_M\zeta(\rho)^2|\alpha|^2\leq C_0\int_{B_1(p)}|ju|^2,$$
which, together with (\ref{betalinf1}), leads us to the estimate
\begin{equation}\label{appbetasupest}
\|\beta\|^2_{L^{\infty}(B_{1/32}(p))}\leq C(n)\left(\int_{B_1(p)}e_{\epsilon}(u)+\|d^*\xi\|_{L^2(M)}^2\right).
\end{equation}

\hspace{6mm} Now, combining (\ref{appdstarxiest}), (\ref{phideltaballest}), and (\ref{appbetasupest}), we see that for any $\delta\leq \frac{1}{32},$ 
\begin{eqnarray*}
\int_{B_{\delta}(p)}|\alpha|^2&\leq &4\int_{B_{\delta}(p)}(|d\varphi|^2+|d^*\xi|^2+|\beta|^2)\\
&\leq &C(n)[\delta^n\int_{B_1(p)}e_{\epsilon}(u)+\int_{B_1(p)}\frac{W(u)}{\epsilon^2}]+4\|d^*\xi\|_{L^2(M)}^2\\
&&+C(n)\delta^n\left(\int_{B_1(p)}e_{\epsilon}(u)+\|d^*\xi\|_{L^2(M)}^2\right)\\
&\leq &C(n)\delta^n\int_{B_1(p)}e_{\epsilon}(u)+C(n)\int_{B_1(p)}\frac{W(u)}{\epsilon^2}\\
&&+C(n,C_1)\int_{B_1(p)}e_{\epsilon}(u)\cdot \int_{B_1(p)}\frac{W(u)}{\epsilon^2}.
\end{eqnarray*}
Finally, we use (\ref{appjualpha}) to see that
$$\int_{B_{\delta}(p)}|ju|^2\leq 2\int_{B_{\delta}(p)}|\alpha|^2+2C(n)\int_{B_1(p)}\frac{W(u)}{\epsilon^2},$$
which, together with the preceding computation and (\ref{nonjparts}), brings us to the desired estimate
\begin{eqnarray*}
\int_{B_{\delta}(p)}e_{\epsilon}(u)&\leq & C(n,C_1)\left(\delta^n+\int_{B_1(p)}\frac{W(u)}{\epsilon^2}\right)\int_{B_1(p)}e_{\epsilon}(u)\\
&&+C(n)\int_{B_1(p)}\frac{W(u)}{\epsilon^2}.
\end{eqnarray*}

\end{proof}

\subsection{The Eta-Ellipticity Result} Having collected the essential lemmas, we turn now to the proof of the $\eta$-ellipticity theorem. Namely, we'll prove the following statement, from which the general version stated in Proposition \ref{etaellip} follows by rescaling:

\begin{prop}\emph{(cf. \cite{BBO}, \cite{LR})} Let $(M^n,g)$ be a compact, orientable $3\leq n$-dimensional manifold satisfying $(\ref{geobounds})$, and let $u\in C^{\infty}(M,\mathbb{C})$ be a solution of the $\epsilon$-Ginzburg-Landau equation $(\ref{gleqnapp})$ for $\epsilon\in (0,\epsilon_0]$, $\epsilon_0=\epsilon_0(M)$. Then there is a constant $\eta(M)>0$ such that if $|u(p)|\leq \frac{1}{2}$, then
\begin{equation}
\int_{B_1(p)}e_{\epsilon}(u)\geq \eta\log (1/\epsilon).
\end{equation}
\end{prop}

\begin{proof} With the proof of Lemma \ref{deltaenlem} out of the way, we can now follow the arguments of \cite{BBO}, with little modification, to arrive at the desired result. We recall those arguments below for the convenience of the reader.

\hspace{6mm} To begin (cf. Lemma III.1 of \cite{BBO}), consider $\delta \in [2\epsilon^{1/4},\frac{1}{32}]$, and let 
$$m+1=\lfloor \frac{\log(\delta/\epsilon^{1/2})}{\log(4/\delta)}\rfloor.$$
Since $\delta^2\geq 4\epsilon^{1/2},$ we see that $m$ is nonnegative, and by definition,
\begin{eqnarray*}
\log(\epsilon^{-1/2})&<&(m+2)[\log 4-\log\delta]-\log\delta\\
&\leq &2(m+3)|\log\delta|
\end{eqnarray*}
(using the fact that $\delta\leq \frac{1}{4}$), so that
\begin{equation}
m+1\geq \frac{1}{3}(m+3)\geq \frac{1}{12}\frac{|\log\epsilon|}{|\log\delta|}.
\end{equation}
Observe next that
$$\bigcup_{j=0}^m((\delta/4)^{-j}\epsilon^{1/2},(\delta/4)^{-j-1}\epsilon^{1/2})\subset (\epsilon^{1/2},\delta),$$
so that, letting $F_{\epsilon}(u,p,r)$ be the monotone quantity given by (\ref{bigfedef}), we have
\begin{eqnarray*}
F_{\epsilon}(u,p,\delta)&\geq &\int_{\epsilon^{1/2}}^{\delta}\frac{d}{dr}F_{\epsilon}(u,p,r)dr\\
&\geq &\Sigma_{j=0}^m\int_{(\delta/4)^{-j}\epsilon^{1/2}}^{(\delta/4)^{-j-1}\epsilon^{1/2}}\frac{d}{dr}F_{\epsilon}(u,p,r)dr\\
&\geq &(m+1)\int_{(\delta/4)^{-j_0}\epsilon^{1/2}}^{(\delta/4)^{-j_0-1}\epsilon^{1/2}}\frac{d}{dr}F_{\epsilon}(u,p,r)dr\\
&\geq &\frac{1}{12}\frac{|\log\epsilon|}{|\log\delta|}\int_{(\delta/4)^{-j_0}\epsilon^{1/2}}^{(\delta/4)^{-j_0-1}\epsilon^{1/2}}\frac{d}{dr}F_{\epsilon}(u,p,r)dr
\end{eqnarray*}
for some $j_0\in \{0,\ldots,m\}$. Recalling the monotonicity formula (\ref{mono2}), from here it's not difficult (cf. \cite{BBO}) to see that one can find 
\begin{equation}
r_0\in [(\delta/4)^{-j_0}\epsilon^{1/2},(\delta/4)^{-j_0-1}\epsilon^{1/2}]\subset (\epsilon^{1/2},\delta)
\end{equation}
such that
\begin{equation}\label{goodradfdiffs}
F_{\epsilon}(u,p,r_0)-F_{\epsilon}(u,p,\delta r_0)\leq C(n)|\log\delta|\frac{F_{\epsilon}(u,p,\delta)}{|\log\epsilon|}
\end{equation}
and
\begin{equation}\label{goodradwest}
r_0^{2-n}\int_{B_{r_0}}\frac{W(u)}{\epsilon^2}\leq C(n)|\log\delta|\frac{F_{\epsilon}(u,p,\delta)}{|\log\epsilon|}.
\end{equation}

\hspace{6mm} Now, set 
\begin{equation}\label{etadef}
\eta:=\frac{\int_{B_1(p)}e_{\epsilon}(u)}{|\log\epsilon|}.
\end{equation}
By definition of $F_{\epsilon}$, we then have $F_{\epsilon}(u,p,1)\leq C\eta|\log\epsilon|$, and by monotonicity, it follows that
\begin{equation}
F_{\epsilon}(u,p,\delta)\leq C\eta|\log\epsilon|
\end{equation}
as well. Plugging this estimate into (\ref{goodradfdiffs}) and (\ref{goodradwest}), we obtain
\begin{equation}\label{goodrad1}
F_{\epsilon}(u,p,r_0)-F_{\epsilon}(u,p,\delta r_0)\leq C\eta|\log\delta|
\end{equation}
and
\begin{equation}\label{goodrad2}
r_0^{2-n}\int_{B_{r_0}}\frac{W(u)}{\epsilon^2}\leq C\eta|\log\delta|.
\end{equation}

\hspace{6mm} Now, since $\delta \leq\frac{1}{32}$, we can apply Corollary \ref{deltaencor} with $r=r_0$, together with (\ref{goodrad2}), to obtain the estimate
\begin{eqnarray*}
\int_{B_{\delta r_0}}e_{\epsilon}(u)&\leq &C\left(\delta^n+r_0^{2-n}\int_{B_{r_0}(p)}\frac{W(u)}{\epsilon^2}\right)\int_{B_{r_0}(p)}e_{\epsilon}(u)\\
&&+C\int_{B_{r_0}(p)}\frac{W(u)}{\epsilon^2}\\
&\leq &C(\delta^n+\eta|\log\delta|)r_0^{n-2}F_{\epsilon}(u,p,r_0)+Cr_0^{n-2}\eta|\log\delta|;
\end{eqnarray*}
or, in terms of $F_{\epsilon}$,
\begin{equation}\label{fdeltarest}
F_{\epsilon}(u,p,\delta r_0)\leq C\delta^{2-n}[(\delta^n+\eta|\log\delta |)F_{\epsilon}(u,p,r_0))+\eta|\log\delta|].
\end{equation}
Combining (\ref{goodrad1}) with (\ref{fdeltarest}), we arrive at the estimate
\begin{equation}\label{goodrad3}
(1-C\eta\delta^{2-n}|\log\delta|-C\delta^2)F_{\epsilon}(u,p,r_0)\leq C\eta|\log\delta|.
\end{equation}

\hspace{6mm} Provided $\eta \leq \frac{1}{(32)^{2(n-2)}},$ we can now choose $\delta=\eta^{\frac{1}{2(n-2)}},$ so that (\ref{goodrad3}) becomes
\begin{equation}
(1-C\eta^{1/2}|\log\eta|-C\eta^{\frac{1}{n-2}})F_{\epsilon}(u,p,r_0)\leq C\eta|\log\eta|,
\end{equation}
and provided $\eta$ is sufficiently small relative to $C(M)$, it follows that
\begin{equation}\label{finalfrest}
F_{\epsilon}(u,p,r_0)\leq \eta^{1/2}.
\end{equation}
By the monotonicity of $F_{\epsilon}(u,p,r)$, since $r_0\geq \epsilon$, we conclude that
\begin{equation}\label{epballavg}
\frac{1}{\epsilon^n}\int_{B_{\epsilon}(p)}W(u)\leq F_{\epsilon}(u,p,\epsilon)\leq \eta^{1/2}
\end{equation}
as well. On the other hand, if $|u|(p)\leq \frac{1}{2}$, then it follows from the gradient estimate
$$|du|\leq \frac{C}{\epsilon}$$
that $|u(x)|\leq \frac{3}{4}$ for all $x\in B_{\epsilon/4C}(p)$, and therefore
$$\frac{1}{\epsilon^n}\int_{B_{\epsilon}(p)}W(u)\geq \frac{1}{\epsilon^n}\int_{B_{\epsilon/4C}(p)}W(3/4)\geq c(M)>0.$$
Thus, by (\ref{epballavg}), we see that if $|u|(p)\leq \frac{1}{2}$ and $\eta$ is sufficiently small relative to $C(M)$, then 
$$\eta^{1/2}\geq c(M)>0.$$
Recalling that $\eta:=\frac{\int_{B_1(p)}e_{\epsilon}(u)}{|\log\epsilon|}$, the proposition follows.
\end{proof}

\end{document}